\documentclass[10pt]{article}
\usepackage{amsmath,amscd}
\usepackage{amssymb,latexsym,amsthm}
\usepackage{color,palatino}
\usepackage[utf8]{inputenc}
\usepackage[T1]{fontenc}
\usepackage{csquotes}
\usepackage{latexsym}
\usepackage{multirow}
\usepackage{lscape}
\usepackage{enumerate}
\usepackage{fancyhdr}
\usepackage{pb-diagram}
\usepackage{amsfonts}
\usepackage{hyperref}
\usepackage{url}
\hypersetup{
  colorlinks=true,
  linkcolor=blue,
  citecolor=black,
  urlcolor=cyan
}

\newtheorem{theorem}{Theorem}[section]
\newtheorem{proposition}[theorem]{Proposition}
\newtheorem{lemma}[theorem]{Lemma}

\newtheorem{corollary}[theorem]{Corollary}

\newtheorem{definition}[theorem]{Definition}
\theoremstyle{definition}
\newtheorem{remark}[theorem]{Remark}

\newtheorem{example}[theorem]{Example}

\newcommand{\K}{\Bbbk}

\newcommand{\C}{\mathbb{C}}

\makeatletter
\def\@roman#1{\romannumeral #1}
\makeatother

\title{\textbf{The PI property of skew PBW extensions}}
\author{James Gómez 
and Claudia Gallego}

\date{}
\begin{document}
\makeatletter
\def\@roman#1{\romannumeral #1}
\makeatother
\maketitle
\begin{abstract}
\noindent
In this article we study the polynomial identity (PI) property of skew PBW extensions. We show that every bijective skew PBW extension over a prime PI-algebra has nontrivial center. This fact allows us to determine, from the known description of the center in several classes of examples, whether such extensions satisfy a polynomial identity. Furthermore, building on results of Brown and Zhang \cite{BrownZhang2022}, we investigate the PI property of certain $\K$-algebras over fields of positive characteristic.
\bigskip

\bigskip

\end{abstract}
\section{Introduction}
\noindent
The earliest research on algebras satisfying polynomial identities (PI algebras) was developed by Dehn in 1922, motivated by determining what types of theorems are valid in non-Archimedean geometries \cite{dehn}. Subsequently, in 1935, Wagner established polynomial identities for matrix algebras and quaternion algebras in the context of the foundations of projective geometry \cite[Chapter 2]{wagner}; shortly thereafter, in 1947, Jacobson and Kaplansky laid the first foundations of what is now known as PI-theory \cite{jacobson, kaplansky}. Along these lines, Pascaud and Valette \cite{pasca}, Cauchon \cite{cauchon}, and Damiano and Shapiro \cite{damiano} established necessary and sufficient conditions for an Ore extension $R[x;\sigma,\delta]$ to satisfy a polynomial identity. In recent years, the study of the PI property has gained renewed interest in noncommutative algebra; more precisely, this property serves as a key tool for classifying the simple modules of an algebra \cite{BeraMukherjee2023, BeraMukherjee2025, BeraMukherjeeUqB2}, and moreover, recent investigations on the Ozone group, as defined by Chan, Gaddis, Won, and Zhang, require the underlying algebra to satisfy the PI property \cite{ChanGaddisWonZhang2025, GaddisYee2025}. A comprehensive survey of the PI property for a wide variety of polynomial-type noncommutative algebras can be found in \cite{gomezgallego1}.\\
\\
In this paper, we establish necessary and sufficient conditions for semicommutative skew PBW extensions to satisfy a polynomial identity. We also obtain such conditions for extensions of the form
\begin{align}\label{skewpbw1}
A=\sigma(\K)\langle x,y\rangle
\text{ subject to the relation }
yx=qxy+\widetilde{q},
\end{align}
where $q\in \K^*$. In addition, we characterize the center of bijective skew PBW extensions satisfying a polynomial identity and defined over a prime $\K$-algebra $R$. This allows us to derive concrete criteria under which skew PBW extensions with known center satisfy the PI property. Finally, we study the PI property for several $\K$-algebras over fields of positive characteristic, within the framework of Brown and Zhang \cite{BrownZhang2022}.\\
\\
The paper is organized as follows. In Section \ref{pbwpi}, we study the PI property for skew PBW extensions. In Theorem \ref{PI-semi-comm}, we prove that a semicommutative skew PBW extension over an integral domain is PI if and only if all the coefficients $c_{i,j}$ are roots of unity. Theorem \ref{PBW2v} establishes the analogous criterion for two-variable extensions over a field, subject to the relation described in (\ref{skewpbw1}). In Theorem \ref{PIprima2}, we show that every bijective skew PBW extension over a prime $\K$-algebra satisfying the PI property has nontrivial center. This allows us to determine, from the known center of certain extensions, whether they are PI. Theorem \ref{necesariaPI} establishes that the scalars $c_{i,j}$ are necessarily roots of unity whenever the extension is PI, although this condition is not sufficient in general. In Section \ref{charpos}, we study the PI property over fields of positive characteristic. Theorem \ref{PI-qc-skewPBW} characterizes the PI property for quasi-commutative skew PBW extensions of the polynomial ring $\K[x]$. Theorems \ref{PI-double-Ore}, \ref{PI-double-Ore2}, and \ref{PI-double-Ore3} deal with double Ore extensions of increasing complexity, characterizing the PI property in each case by finite-order conditions on the restrictions of the endomorphisms to the center. As a consequence, whenever the PI property holds, each of these algebras is a finitely generated module over its center.\\
\\
Recall that given  a commutative ring $R$ and an $R$-algebra $A$, a \textbf{polynomial identity} (PI) for $A$ is a polynomial
\[
f(x_1, x_2, \dots, x_n) \in R\langle x_1, \dots, x_n \rangle
\]
such that
\[
f(a_1, a_2, \dots, a_n) = 0, \quad \text{for all }\quad a_1, a_2, \dots, a_n \in A.
\]
An algebra for which some monic polynomial identity exists is called a
\textbf{PI-algebra} (here, the fact that $f$ is monic means that among the monomials of highest total degree which appear in $f$, at least one has coefficient $1$). Alternatively, one says that the algebra \textbf{satisfies the PI property}.
In addition, if $A$ is an arbitrary ring, a \textbf{polynomial identity} for the ring $A$ is a  polynomial $f \in \mathbb{Z}\langle x_1,\ldots,x_n\rangle$
such that
\[
f(a_1,\ldots,a_n)=0 \quad \text{for all}\quad a_1,\ldots,a_n \in A.
\]
A ring for which some monic polynomial identity exists is called a
\textbf{PI-ring}. Additionally, a PI-ring $A$ is said to have \textbf{minimal degree} $d$  if $d$ is the least possible degree of a monic polynomial identity of $A$. By \cite[Proposition 13.1.9] {mcconnell} this identity can be chosen to be multilinear. 
\section{PI property on skew PBW extensions}\label{pbwpi}
\noindent
In this section we study the PI property for skew PBW extensions. We first recall the basic definitions and some structural facts concerning skew PBW extensions, with special emphasis on the semicommutative and bijective cases. We then prove a criterion for semicommutative skew PBW extensions over an integral domain to satisfy a polynomial identity, in terms of the corresponding commutation parameters. Afterwards, we consider several classes of skew PBW extensions which are not necessarily semicommutative. In this setting, the center plays a decisive role: we show that a bijective skew PBW extension over a prime \(\K\)-algebra satisfying a polynomial identity cannot have center equal to \(\K\). This observation provides a useful obstruction to the PI property and allows us to treat a number of examples whose centers are known in the literature.
\subsection{Definition and some examples}
\noindent
In this section we recall the definition of skew PBW (Poincaré-Birkhoff-Witt) extensions first introduced in \cite{GallegoLezama}, and investigate the PI property for semicommutative skew PBW extensions. Our main goal is to establish a characterization of this property in terms of the commutation parameters of the extension.
\begin{definition}[{\cite[Definition 1]{GallegoLezama}}]\label{pbwt}
Let $R$ and $A$ be rings. It is said that $A$ is a \textbf{skew PBW extension of $R$} (also called
a \textbf{$\sigma$-PBW extension} of $R$) if the following conditions hold:
\begin{enumerate}
\item[\rm (i)]$R$ is a subring of $A$.
\item[\rm (ii)]There exist finitely many elements $x_1,\dots ,x_n\in A$ such that $A$ is a left $R$-free module with basis
\begin{center}
${\rm Mon}(A):= \{x^{\alpha}=x_1^{\alpha_1}\cdots x_n^{\alpha_n}\mid \alpha=(\alpha_1,\dots
,\alpha_n)\in \mathbb{N}^n\}$.
\end{center}
\item[\rm (iii)]For every $1\leq i\leq n$ and $r\in R\setminus\{0\}$ there exists $c_{i,r}\in R\setminus\{0\}$ such that
\begin{equation}\label{sigmadefinicion1}
x_ir-c_{i,r}x_i\in R.
\end{equation}
\item[\rm (iv)]For every $1\leq i,j\leq n$ there exists $c_{i,j}\in R\setminus\{0\}$ such that
\begin{equation}\label{sigmadefinicion2}
x_jx_i-c_{i,j}x_ix_j\in R+Rx_1+\cdots +Rx_n.
\end{equation}
\end{enumerate}
Under these conditions, the notation $A:=\sigma(R)\langle x_1,\dots ,x_n\rangle$ is used and $R$ is the \textbf{ring of coefficients} of the extension.
\end{definition}
The following property, besides justifying the notation introduced for skew PBW extensions, will prove useful in establishing several of the results of this section.
\begin{proposition}[{\cite[Proposition 3]{GallegoLezama}}]\label{prop23} Let $A$ be a skew PBW extension of $R$. Then, for every $1\leq i \leq n$, there exist an injective ring endomorphism $\sigma_i : R \rightarrow R$ and a $\sigma_i$-derivation $\delta_i: R \rightarrow R$  such that $$x_i r = \sigma_i(r)x_i + \delta_i(r),$$
 for all $r \in R$.
\end{proposition}
We now introduce four important types of skew PBW extensions.
\begin{definition} [{\cite[Definition 4]{GallegoLezama}, \cite[Definition 2.5]{Suarez2017}}]
 Let $A$ be a skew PBW extension of $R$. Denote $\Sigma := \{\sigma_1,\ldots,\sigma_n\}$ and $\Delta := \{\delta_1,\ldots,\delta_n\}$, where $\sigma_i$ and $\delta_i$, for $1 \le i \le n$, are as in Proposition \ref{prop23}.
\begin{enumerate}[\rm (i)]
 \item The extension $A$ is said to be  \textbf{bijective} if $\sigma_i$ is bijective for each $\sigma_i \in \Sigma$  and $c_{i,j}$ is invertible for all $1 \le i < j \le n$.
\item The extension $A$ is said to be \textbf{quasi-commutative} if conditions (iii) and (iv) of Definition \ref{pbwt} are replaced by the following:
\begin{enumerate}
\item[\rm (iiia)] For each $1\leq i \leq n$ and $r \in R\setminus\{0\}$, there exists $c_{i,r} \in R\setminus\{0\}$ such that 
\begin{align}
x_ir = c_{i,r}x_i.
\end{align}
\item [\rm (iva)] For all $1\leq i,j \leq n$, there exists $c_{i,j} \in R\setminus\{0\}$ such that 
\begin{align}\label{cuasi4}
x_jx_i = c_{i,j}x_ix_j.
\end{align}
\end{enumerate}
\item An element $r\in R$ is called \textbf{constant} if $\sigma_i(r) = r$ and $\delta_i(r) = 0$  for all $1 \leq i \leq n$. In these conditions, the ring $A$ is said to be \textbf{constant} if every element of $R$ is constant.
\item The extension $A$ is said to be \textbf{semicommutative} if it is both quasi-commutative and constant.
\end{enumerate}
\end{definition}
Some examples of particular interest are introduced below.
\begin{example}[Iterated Ore extensions as skew PBW extensions]
Let $R[x_1;\sigma_1,\delta_1]\cdots[x_n;\sigma_n,\delta_n]$ be an \textit{iterated skew polynomial ring of injective type}, i.e., if the following conditions hold:

\begin{itemize}
    \item[] For $1 \leq i \leq n$, the endomorphism $\sigma_i$ is injective.
    \item[] For every $r \in R$ and $1 \leq i \leq n$, one has  $\sigma_i(r), \delta_i(r) \in R$.
    \item[] For $i < j$, the condition $\sigma_j(x_i) = cx_i + d$ holds, with $c, d \in R$ and $c$ a left invertible element from $R$.
    \item[] For $i < j$, it follows that $\delta_j(x_i) \in R + Rx_1 + \cdots + Rx_i$.
\end{itemize}

\noindent Then, $R[x_1;\sigma_1,\delta_1]\cdots[x_n;\sigma_n,\delta_n]$ is a skew $PBW$ extension. 
Under these conditions we have
\[
R[x_1;\sigma_1,\delta_1]\cdots[x_n;\sigma_n,\delta_n] = \sigma(R)\langle x_1,\ldots,x_n\rangle.
\]

\noindent In particular, any \textit{Ore extension} $R[x_1;\sigma_1,\delta_1]\cdots[x_n;\sigma_n,\delta_n]$ \textit{of injective type}, i.e., those Ore extensions in which all endomorphisms 
$\sigma_i$ are injective for $1 \leq i \leq n$,  are skew PBW extensions, see \cite[pg. 8]{skewbook}. Indeed, in Ore extensions, for every $r \in R$ and $1 \leq i \leq n$, one has that $\sigma_i(r), \delta_i(r) \in R$, and for $i < j$, it follows that $\sigma_j(x_i) = x_i$ and $\delta_j(x_i) = 0$. An important subclass of Ore extensions of injective type are the \textit{Ore algebras of injective type}, i.e., when $R = \K[t_1,\ldots,t_m]$, $m \geq 0$, and $\sigma_i, \delta_i$ are $\K$-linear. Thus, we have
\[
\K[t_1,\ldots,t_m][x_1;\sigma_1,\delta_1]\cdots[x_n;\sigma_n,\delta_n] 
= \sigma(\K[t_1,\ldots,t_m])\langle x_1,\ldots,x_n\rangle.
\]

\end{example}
The following are some representative examples of semicommutative skew PBW extensions.
\begin{example}\label{Ejemplosemi}
\begin{enumerate}[(1)]
\item (Classical polynomial ring.)
Let $A = R[t_1, \ldots, t_n]$ be the classical polynomial ring. Since $t_i r = r t_i$  and $t_i t_j = t_j t_i$, for all $r \in R$ and $1 \leq i, j \leq n$, we have $A \cong \sigma(R)\langle t_1, \ldots, t_n \rangle$ as a semicommutative skew PBW extension of $R$, with $\sigma_i = \mathsf{id}_R$ and $\delta_i = 0$, for all $1\leq i\leq n$.
\item (Sklyanin algebra — particular case).
The Sklyanin algebra is defined as 
\[
S = \K\langle x, y, z \rangle / (ayx + bxy + cz^2,\; axz + bzx + cy^2,\; azy + byz + cx^2),
\]
where $a$, $b$, $c \in \K$. If $c = 0$ and $a$, $b \neq 0$, the defining relations reduce to
 \[
 yx = -\frac{b}{a}xy, \qquad zx = -\frac{b}{a}xz, \qquad zy = -\frac{b}{a}yz,
 \]
 so that $S \cong \sigma(\K)\langle x, y, z \rangle$ is a semicommutative skew PBW extension of $\K$.
\item (Multiparameter quantum affine space).
Let $n \geq 1$ and let $\mathbf{q} = (q_{ij})_{n \times n}$  be a matrix with entries in $\K$ satisfying $q_{ii} = 1$ and $q_{ij}q_{ji} = 1$, for all $1 \leq i, j \leq n$. The \textit{multiparameter quantum affine space} of dimension $n$, denoted $\mathcal{O}_{\mathbf{q}}(\K^n)$, is the $\K$-algebra generated by $x_1, \ldots, x_n$  subject to the relations
 \[
 x_j x_i = q_{ij} x_i x_j, \qquad 1 \leq i, j \leq n.
 \]
This algebra is a semicommutative skew PBW extension of $\K$.
\item (Multiplicative analogue of the Weyl algebra).
This algebra, denoted $\mathcal{O}_n(\lambda_{ji})$, is generated by $x_1, \ldots, x_n$ subject to the relations
\[
x_j x_i = \lambda_{ji} x_i x_j, \qquad 1 \leq i < j \leq n, \quad \lambda_{ji} \in \K \setminus \{0\},
 \]
so that $\mathcal{O}_n(\lambda_{ji}) \cong \sigma(\K)\langle x_1, \ldots, x_n \rangle$ is a semicommutative skew PBW extension of $\K$. For $n=2$, this algebra is known as the \textit{quantum plane}.
\item (Three-dimensional skew polynomial algebra).
Let $A$ be the algebra generated by $x$, $y$, $z$  subject to the relations
\[
 yz - \alpha zy = \lambda, \qquad zx - \beta xz = \mu, \qquad xy - \gamma yx = \nu,
 \]
 where $\lambda$, $\mu$, $\nu \in \K + \K x + \K y + \K z$, $\alpha$, $\beta$, $\gamma \in \K^*$, and such that the set of standard monomials $\{x^i y^j z^l \mid i,j,l \ge 0\}$ forms a $\K$-basis of $A$. If the three scalars $\alpha$, $\beta$, $\gamma$
 are pairwise distinct, i.e., if $|\{\alpha, \beta, \gamma\}| = 3$, the relations reduce to
 \[
 yz = \alpha zy, \qquad zx = \beta xz, \qquad xy = \gamma yx,
 \]
and hence $A \cong \sigma(\K)\langle x, y, z \rangle$ is a semicommutative skew PBW extension of $\K$.
\end{enumerate}
 \end{example}

 \begin{remark}\label{weyl}
The multiplicative analogue of the Weyl algebra can also be realized as a skew 
PBW extension over the ring $\K[x_1]$ (see \cite{LezamaReyes2014}, 
Example~3.5-(b)),
\[
\mathcal{O}_n(\lambda_{ji}) \cong \sigma(\K[x_1])\langle x_2, \ldots, x_n \rangle.
\]
Both skew PBW extensions are isomorphic as $\K$-algebras; however, 
$\sigma(\K[x_1])\langle x_2, \ldots, x_n \rangle$ is not semicommutative. This is 
due to the fact that the semicommutativity property depends not only on the 
extension itself, but also on the base ring and the actions $(\sigma_i, \delta_i)$; 
in other words, the isomorphism holds between the extensions as algebras, but not 
between the pairs $(A, R)$ consisting of each extension together with its 
respective base ring.
\end{remark}

\begin{proposition}\label{centro3}
Let $A := \sigma(R)\langle x_1, \dots, x_n\rangle$ be a \textit{semicommutative} 
skew PBW extension of $R$. Suppose that for all $1 \leq i, j \leq n$, with 
$i \neq j$, the coefficients $c_{i,j}$ (as in \eqref{cuasi4}) are roots of unity, 
and let
\begin{align}\label{defofl}
L := \mathsf{lcm}\{\, \mathsf{ord}(c_{i,j}) \mid i \neq j \,\}.
\end{align}
Then, the powers $x_1^{L}, x_2^{L}, \dots, x_n^{L}$ are central elements of $A$.
\end{proposition}

\begin{proof}
Since $A$ is a semicommutative skew PBW extension, we have
\[
x_i r = r x_i \quad (r \in R,\ 1 \le i \le n)
\qquad \text{and} \qquad
x_j x_i = c_{i,j}\, x_i x_j \quad (1 \le i, j \le n).
\]
Fix $i$, with $1 \leq i \leq n$. An induction argument on $m$ shows that 
$x_i^{m} r = r x_i^{m}$ for all $r \in R$ and $m \geq 1$; in particular, 
$x_i^{L}$ commutes with every element of $R$. Similarly, an induction on $m$ 
establishes that
\[
x_j x_i^{m} = c_{i,j}^{\,m}\, x_i^{m} x_j, \qquad \text{for all } m \geq 1.
\]
Since $L$ is a multiple of the order of each $c_{i,j}$, we have $c_{i,j}^L = 1$, 
and therefore $x_j x_i^{L} = x_i^{L} x_j$ for all $j$. Hence $x_i^{L}$ commutes 
with all generators $x_j$ and with all elements of $R$, so that 
$x_i^{L} \in Z(A)$.
\end{proof}
Recall that a finite-dimensional $\K$-algebra $A$ is a \textbf{central simple algebra} over $\K$ if $Z(A)=\K$ and the only two-sided ideals are  $\textbf{0}$ and $A$; equivalently, if $A\cong M_n(D)$ --$\K$-algebra isomorphism--, for some $n\geq 1$ and a division $\K$-algebra $D$, with $Z(D)=\K$, see e.g. \cite{mcconnell}.  
\begin{theorem}[Posner's Theorem, {\cite[Theorem 13.6.5]{mcconnell}}]\label{posner}
Let $A$ be a prime PI-ring with center $Z$ and minimal degree $d$. Let 
$S := Z \setminus \{0\}$, $Q := AS^{-1}$, and $\mathbb{F}:= ZS^{-1}$ be the field of 
fractions of $Z$. Then $Q$ is a central simple algebra with center $\mathbb{F}$, and
\[
\mathsf{dim}_{\mathbb{F}}(Q) = \left(\frac{d}{2}\right)^{2}.
\]
Under these conditions,  the \textbf{PI-degree} of $A$ is defined as $n:=\frac{d}{2}$.
\end{theorem}

The following result establishes necessary and sufficient conditions for a 
semicommutative skew PBW extension to be a PI-ring. The proof makes use of 
Posner's Theorem.

\begin{theorem}\label{PI-semi-comm}
Let $A := \sigma(R)\langle x_1,\dots,x_n\rangle$ be a semicommutative skew PBW 
extension of an integral domain $R$. Then $A$ is a PI-ring if and only if for 
all $1 \le i < j \le n$, the coefficients $c_{i,j}$ are roots of unity.
\end{theorem}

\begin{proof}
Suppose first that $A$ is a PI-ring. Fix $1 \le i < j \le n$ and let 
$q := c_{i,j}$. Note that this $q$ is an invertible element of $R$. Since $A$ is semicommutative, we have 
$x_k r = r x_k$ for all $r \in R$ and all $1 \le k \le n$; in particular, 
$R \subseteq Z(A)$. Hence the subring
\[
B := R\langle x_i, x_j \rangle \subseteq A
\]
is generated over $R$ by $x_i$ and $x_j$ subject to the single relation 
$x_j x_i = q\, x_i x_j$, so that
\[
B \cong R\langle X,Y\rangle/(YX - qXY)  
\cong R[x][y;\sigma],
\]
where $\sigma:R[x]\to R[x]$ is defined by $\sigma(x) := qx$ and $\sigma(r):=r$, for all $r\in R$.
Since $\sigma$ is injective and $R[x]$ is a domain, the ring
$R[x][y;\sigma]$ is a domain and hence a prime ring; that is, $B$ is a prime ring. 
Moreover, since $B \subseteq A$ and $A$ satisfies the PI property, $B$ also 
satisfies the PI property.

Let $T:=Z(B) \setminus \{0\}$, $S := R \setminus \{0\}$ and $\K:= S^{-1}R$  the field of fractions of $R$.  Note that $T$ is 
a multiplicatively closed set of regular elements of $B$ and we can consider the ring of fractions $Q_c(B) := T^{-1}B$. By Posner's Theorem (see Theorem \ref{posner}), such  
$Q_c(B)$ is a central simple algebra of finite dimension over the field of 
fractions of $Z(B)$; in particular, $Q_c(B)$ is a PI-ring. On the other hand, since 
$S \subseteq T$, we have $S^{-1}B \hookrightarrow T^{-1}B = Q_c(B)$, so 
$S^{-1}B$ is a subring of a PI-ring and therefore is itself a PI-ring. 
Furthermore,
\[
S^{-1}B \cong \K\langle X,Y\rangle/(YX - qXY) = \mathcal{O}_q(\K^2).
\]
By \cite[Section~7.1]{DeConciniProcesi1993}, this ring $\mathcal{O}_q(\K^2)$ satisfies the 
PI property if and only if $q$ is a root of unity in $\K$. Consequently, there 
exists $L \ge 1$ such that $q^L = 1$ in $\K$. Considering that $R$ can be seen as a subring of $\K$,  it follows that $q^L = 1$ in $R$, and hence $q = c_{i,j}$ is a root 
of unity in $R$. 

Conversely, suppose that all $c_{i,j}$ are roots of unity, and let 
$L$ as in (\ref{defofl}). 
By Proposition~\ref{centro3}, the power $x_i^L$ is a central element of $A$, for all $1 \le i \le n$. If $C$ denotes the subalgebra of $A$ generated by $R$ and $x_1^L, \dots, x_n^L$, then 
$C \subseteq Z(A)$. Moreover, since every element of $A$ is an $R$-linear combination of standard 
monomials $x_1^{m_1} \cdots x_n^{m_n}$, writing $m_i = a_i + Lt_i$, with 
$0 \le a_i < L$ and $t_i \in \mathbb{N}$, we obtain
\[
x_1^{m_1} \cdots x_n^{m_n}
= (x_1^L)^{t_1} \cdots (x_n^L)^{t_n}\, x_1^{a_1} \cdots x_n^{a_n}.
\]
Therefore, the extension $A$ is a $C$-module generated by the finite set
\[
\mathcal{B} := \bigl\{ x_1^{a_1} \cdots x_n^{a_n} \mid 0 \le a_i < L \bigr\}.
\]
Thus, we have that $A$ is finitely generated as a module over the commutative central subring 
$C$, and \cite[\S 13.1.13, Corollary]{mcconnell} implies that $A$ is a PI-ring.
\end{proof}
\begin{remark}
The following observations concerning the proof of the previous theorem are worth \-no\-ting.
\begin{enumerate}[(1)]
\item If, in addition, the integral domain $R$ is assumed to be Noetherian, then the fact that \(S^{-1}B\) is a PI-ring follows directly from Lemma 2.1-(d) of \cite{Bell1989}.
\item Another way to ensure that \(S^{-1}B\) is a PI-ring is to assume that the polynomial identity satisfied by \(A\) is multilinear. Indeed, under this assumption, the PI property is preserved under localization: if $\frac{b_1}{s_1},\ldots,\frac{b_t}{s_t}\in S^{-1}B$ and $f$ is a multilinear polynomial, then  
\begin{align*}
f\left(\frac{b_1}{s_1},\ldots,\frac{b_t}{s_t}\right)=  \frac{f(b_1,\ldots,b_t)}{s_1\ldots s_t}.  
\end{align*}
As $f(b_1,\ldots,b_t)=0$ in $B$, it follows that $f\left(\frac{b_1}{s_1},\ldots,\frac{b_t}{s_t}\right)=0$ in $S^{-1}B$.
\item In Example \ref{Ejemplosemi}, all semicommutative skew PBW extensions are defined over a field \(\K\). Therefore, if in the previous theorem the assumption that \(R\) is an integral domain is replaced by the stronger assumption that \(R=\K\) is a field, the result follows immediately. In fact, semicommutative skew PBW extensions of the form
\[
\sigma(\K)\langle x_1,\ldots,x_n\rangle
\]
are isomorphic to the multiparameter quantum affine space \(\mathcal{O}_{\lambda}(\K^n)\) generated by \(x_1,\ldots,x_n\), and the corresponding characterization of the PI property was established in \cite{Haynal2008}.
\end{enumerate}
\end{remark}
In \cite{LezamaReyes2014} it is shown that, given a skew PBW extension 
$A := \sigma(R)\langle x_1,\dots,x_n\rangle$, one can construct an associated 
quasi-commutative skew PBW extension whose structure parameters coincide with 
those of the original extension.

\begin{proposition}[{\cite[Proposition 2.1]{LezamaReyes2014}}]\label{prop29}
Let $A$ be a skew PBW extension of $R$. Then there exists a quasi-commutative 
skew PBW extension $A^{\sigma}$ of $R$ in $n$ variables $z_1,\ldots,z_n$, 
defined by the relations
\[
z_i r = c_{i,r}\, z_i, \qquad z_j z_i = c_{i,j}\, z_i z_j, \quad 1 \le i,j \le n,
\]
where $c_{i,r}$ and $c_{i,j}$ are the same constants as those appearing in the 
defining relations of $A$. Moreover, if $A$ is bijective, then $A^{\sigma}$ is 
bijective as well.
\end{proposition}
\begin{example}[Diffusion algebra]
The diffusion algebra $\mathcal{D}$ is generated by $\{D_i, x_i \mid 1 \le i \le n\}$ 
over a field $\K$, subject to the relations
\[
\begin{aligned}
&x_i x_j = x_j x_i, \qquad x_i D_j = D_j x_i \quad (1 \le i,j \le n),\\
&c_{i,j} D_i D_j - c_{j,i} D_j D_i = x_j D_i - x_i D_j \quad (i < j),
\end{aligned}
\]
where $c_{i,j}, c_{j,i} \in \K^{*}$. From the latter, one obtains that
\[
D_j D_i - \left(\frac{c_{i,j}}{c_{j,i}}\right) D_i D_j
= \frac{1}{c_{j,i}}\,(x_i D_j - x_j D_i)
\in RD_i + RD_j
\subseteq R + RD_1 + \cdots + RD_n,
\]
so condition (iv) of Definition \ref{pbwt} holds for the pair $(D_i, D_j)$
\[
D_j D_i - \tilde{c}_{i,j}\, D_i D_j \in R + RD_1 + \cdots + RD_n,
\]
with $ \widetilde{c}_{i,j} := \frac{c_{i,j}}{c_{j,i}} \in \K^{*}$.
Therefore, the algebra $\mathcal{D}$ is isomorphic to a bijective skew PBW extension  
$A$ of $\K[x_1,\ldots,x_n]$ that is not 
quasi-commutative. However, the  quasi-commutative skew PBW extension 
$A^{\sigma}$ 
defined in Proposition \ref{prop29} is 
semicommutative. In consequence, by Theorem~\ref{PI-semi-comm}, the algebra $A^{\sigma}$ satisfies the PI 
property if and only if each $\widetilde{c}_{i,j}$ is a root of unity, for all 
$1 \le i < j \le n$.
\end{example}
There are numerous skew PBW extensions over a field $\K$, that is, of the form 
$A := \sigma(K)\langle x_1,\dots,x_n\rangle$; for instance, the additive version 
of the Weyl algebra, the quantum algebra $U'(\mathfrak{so}(3,\K))$, the dispin 
algebra $U(\mathfrak{osp}(1,2))$, and the $q$-deformed Heisenberg algebra 
$H_n(q)$, among many others (see \cite{LezamaReyes2014}).

\begin{corollary}
Let $A$ be a skew PBW extension over a field $\K$. Then the extension $A^{\sigma}$ 
defined in Proposition~\ref{prop29} satisfies the PI property if and only if for 
all $1 \le i < j \le n$, the coefficients $c_{i,j}$ are roots of unity.
\end{corollary}

\begin{proof}
By Proposition~\ref{prop29}, there exists a quasi-commutative skew PBW extension 
$A^{\sigma}$ of $\K$ in variables $z_1,\ldots,z_n$. Since $A$ is defined over the 
field $\K$, we have $\sigma_i(k) = k$ and $\delta_i(k) = 0$ for all $k \in \K$ and 
all $1 \le i \le n$; consequently, the extension $A^{\sigma}$ is constant and, in particular, 
semicommutative. Applying Theorem \ref{PI-semi-comm}, we conclude that $A^{\sigma}$ 
satisfies the PI property.
\end{proof}
\begin{example}[Double Ore extension over a field $\K$ with $p_{11}=0$]
Let $\K$ be a field. The trimmed right double Ore extension 
$B = \K_P[y_1,y_2;\sigma]$, with parameter $p_{12} \in \K^{*}$ and $p_{11}=0$, 
is the $\K$-algebra generated by $y_1$ and $y_2$ subject to the relations
\[
\begin{pmatrix}y_1\\[2pt]y_2\end{pmatrix}\lambda
= \lambda \begin{pmatrix}y_1\\[2pt]y_2\end{pmatrix}\text{, for $\lambda \in \K$, and }
y_2 y_1 = p_{12}\, y_1 y_2.
\]
Hence, the algebra  $B$ is isomorphic to a semicommutative skew PBW extension 
$B \cong \sigma(\K)\langle y_1, y_2\rangle$ of $\K$ and, by Theorem~\ref{PI-semi-comm}, 
$B$ satisfies the PI property if and only if $p_{12}$ is a root of unity.
\end{example}

We now present examples of skew PBW extensions that, while not semicommutative, 
satisfy the PI property under suitable conditions on their parameters. In all 
cases, the algebras under consideration admit a presentation as iterated Ore 
extensions or skew polynomial rings.

\begin{example}
In \cite{LopesSuarezSuarez2025}, Proposition 2.2, the two-parameter Heisenberg 
enveloping algebra $H_{p,q}$ is presented as a bijective skew PBW extension 
over $\K[t]$,
\[
H_{p,q} \cong \sigma(\K[t])\langle x, y\rangle.
\]
The endomorphisms and derivations of Proposition~\ref{prop23} are given by 
$\sigma_x(t) = pt$, $\delta_x(t) = 0$, $\sigma_y(t) = p^{-1}t$, and 
$\delta_y(t) = 0$. Note that $H_{p,q}$ is not a semicommutative skew PBW 
extension, since the element $t$ is not constant in the extension; however, 
when the parameters $p$ and $q$ are roots of unity, $H_{p,q}$ satisfies the 
PI property, see \cite{gomezgallego1}.
\end{example}

\begin{example}
The two-parameter quantum matrix algebra $M_2(\alpha,\beta)$ admits a 
presentation as an iterated Ore extension of the form
\[
\K[X_{11}]\,[X_{12};\sigma_{12}]\,[X_{21};\sigma_{21}]\,[X_{22};\sigma_{22},\delta_{22}],
\]
where $\sigma_{12}, \sigma_{21}, \sigma_{22}$ are $\K$-linear automorphisms and 
$\delta_{22}$ is a $\K$-linear $\sigma_{22}$-derivation. Consequently, 
$M_2(\alpha,\beta)$ can be realized as a skew PBW extension of the form
\[
M_2(\alpha,\beta) \cong \sigma(\K)\langle X_{11}, X_{12}, X_{21}, X_{22}\rangle.
\]
Note that $M_2(\alpha,\beta)$ is not a quasi-commutative skew PBW extension; 
however, when the parameters $\alpha$ and $\beta$ are roots of unity, this 
algebra satisfies the PI property, see \cite{gomezgallego1}.
\end{example}
\subsection{Skew PBW extensions in two variables over a field}
\noindent
Throughout this subsection we assume that $\K$ is a field of characteristic zero. 
We consider skew PBW extensions in two variables over $\K$, that is, $\K$-algebras 
$A$ generated by $x$ and $y$ subject to the relation
\begin{equation}\label{eq:relacion-2-1}
  yx \;=\; q_1\,xy + q_2\,x + q_3\,y + q_4,
\end{equation}
where $q_i \in \K$ for $i=1,2,3,4$ and $q_1 \neq 0$.

\begin{lemma}[{\cite[Lemma 2.1]{LezamaVenegas2020}}]\label{centropbw2v}
For $n \ge 2$, the following identities hold in $A$:
\begin{align}
yx^{\,n}
&= \sum_{j=0}^{n} \binom{n}{j}\, q_{1}^{\,n-j} q_{3}^{\,j}\, x^{\,n-j} y
   + q_{2}\!\left(\sum_{i=0}^{n-1} q_{1}^{\,i}\right) x^{\,n} \notag\\
&\quad + \Bigg[
      q_{4}\!\left(\sum_{i=0}^{n-1} q_{1}^{\,i}\right)
    + q_{2}q_{3}\!\left(\sum_{i=0}^{n-2} \binom{i+1}{1}\, q_{1}^{\,i}\right)x^{n-1}
    + \cdots \notag\\
&\qquad\qquad
    + q_{4}q_{3}^{\,j-1}\!\left(\sum_{i=0}^{n-j} \binom{i+j-1}{\,j-1\,}\, q_{1}^{\,i}\right)
    + q_{2}q_{3}^{\,j}\!\left(\sum_{i=0}^{n-j-1} \binom{i+j}{\,j\,}\, q_{1}^{\,i}\right)
    \Bigg] x^{\,n-j}; \label{eq:ident1}
\end{align}
\begin{align}
y^{\,n}x
&= \sum_{j=0}^{n} \binom{n}{j}\, q_{1}^{\,n-j} q_{2}^{\,j}\, x\, y^{\,n-j}
   + q_{3}\!\left(\sum_{i=0}^{n-1} q_{1}^{\,i}\right) y^{\,n} \notag\\
&\quad + \Bigg[
      q_{4}\!\left(\sum_{i=0}^{n-1} q_{1}^{\,i}\right)
    + q_{2}q_{3}\!\left(\sum_{i=0}^{n-2} \binom{i+1}{1}\, q_{1}^{\,i}\right)y^{n-1}
    + \cdots \notag\\
&\qquad\qquad
    + q_{4}q_{2}^{\,j-1}\!\left(\sum_{i=0}^{n-j} \binom{i+j-1}{\,j-1\,}\, q_{1}^{\,i}\right)
    + q_{2}^{\,j}q_{3}\!\left(\sum_{i=0}^{n-j-1} \binom{i+j}{\,j\,}\, q_{1}^{\,i}\right)
    \Bigg] y^{\,n-j}. \label{eq:ident2}
\end{align}
\end{lemma}

In the particular case $q_2 = q_3 = 0$, the above identities reduce to
\[
yx^{n} = q_{1}^{\,n} x^{n} y + q_{4}\!\left(\sum_{i=0}^{n-1} q_{1}^{\,i}\right) x^{n-1},
\qquad
y^{n}x = q_{1}^{\,n} x y^{n} + q_{4}\!\left(\sum_{i=0}^{n-1} q_{1}^{\,i}\right) y^{n-1}.
\]

\begin{proposition}\label{centroPBW2v}
Let $A = \sigma(\K)\langle x,y\rangle$ be a skew PBW extension over $\K$ subject 
to the relation
\[
yx = q_{1}xy + q_{4}, \qquad q_{1} \in \K^{*} \quad (q_{2}=q_{3}=0).
\]
If $q_{1}$ is an $n$-th root of unity with $n \ge 2$, then $x^{n}$ and $y^{n}$ 
are central elements of $A$. In particular, $\K[x^{n}, y^{n}] \subseteq Z(A)$.
\end{proposition}

\begin{proof}
Since $q_{2} = q_{3} = 0$, Lemma~\ref{centropbw2v} yields the identities
\[
yx^{n} = q_{1}^{\,n}x^{n}y + q_{4}\!\left(\sum_{i=0}^{n-1}q_{1}^{\,i}\right)x^{n-1},
\qquad
y^{n}x = q_{1}^{\,n}xy^{n} + q_{4}\!\left(\sum_{i=0}^{n-1}q_{1}^{\,i}\right)y^{n-1}.
\]
Since $q_{1}^{n} = 1$ and $n \ge 2$, we have
\[
\sum_{i=0}^{n-1}q_{1}^{\,i} = \frac{1-q_{1}^{n}}{1-q_{1}} = 0.
\]
Therefore $yx^{n} = x^{n}y$ and $y^{n}x = xy^{n}$, so that $x^{n}$ and $y^{n}$ 
commute with both generators $x$ and $y$ and are hence central elements of $A$.
\end{proof}

As was noted before, the quasi-commutative extension $A^{\sigma}$ constructed in
Proposition~\ref{prop29} inherits the scalar parameters $c_{i,r}$
and $c_{i,j}$ that define~$A$; in particular, when
$A=\sigma(\K)\langle x,y\rangle$ is defined by the relation
$yx=q_1 xy+q_4$, the associated extension $A^{\sigma}$ is determined
solely by the constant $q_1$. This makes it possible to characterize
the PI property of $A$ entirely in terms of this single parameter,
as the following result shows.

\begin{theorem}\label{PBW2v}
Let \(A=\sigma(\K)\langle x,y\rangle\) be a skew PBW extension
over~\(\K\) subject to the relation
\[
  yx = q_{1}xy + q_{4},\qquad q_{1}\in \K^{*}.
\]
Then $A$ is a PI-algebra if and only if $q_1$ is a primitive $n$-th
root of unity for some $n\ge 2$.
\end{theorem}

\begin{proof}
Suppose first that $q_1$ is an $n$-th root of unity with $n\ge 2$.
By Proposition~\ref{centroPBW2v}, one has that $x^n,y^n\in Z(A)$, so that
$Z_0:=\K[x^n,y^n]\subseteq Z(A)$. Consequently, the extension $A$ is generated as
a $Z_0$-module by the finite set $\{x^i y^j\mid 0\le i,j<n\}$,
and hence $A$ satisfies the PI property.

\noindent
Conversely, suppose that $A$ is a PI-algebra and that $q_1$ is not
a root of unity.
\begin{enumerate}[(a)]
  \item If $q_4\neq 0$, note that $A$ may be realized as the
    Ore extension
    \[
      A = \K[x][y;\,\sigma,\,\delta],\qquad
      \sigma(x)=q_1 x,\quad \delta(x)=q_4,
    \]
    where $\sigma$ is an automorphism of $\K[x]$ and $\delta$ is a
    $\sigma$-derivation. Hence $A$ is a prime domain whose center
    satisfies $Z(A)=\K$ by \cite[Theorem 2.2-(i)]{LezamaVenegas2020}.
    Since $A$ is a prime PI-domain, Posner's Theorem
   (see Theorem \ref{posner}) implies that $Q:=AS^{-1}$, with $S=\K^*$,
    is a finite-dimensional central simple algebra over
    $\mathbb{F}:=Z(A)S^{-1}=\K$; in particular, $\mathsf{dim}_{\K}(A)$ is finite. This
    contradicts the fact that $A$ admits the infinite $\K$-basis
    $\{x^i y^j\mid i,j\ge 0\}$.
\item If $q_4=0$, then $A=\sigma(\K)\langle x,y\rangle$ is
    a semicommutative skew PBW extension over $\K$, and
    Theorem~\ref{PI-semi-comm} forces $q_1$ to be a root of unity,
    contradicting the hypothesis.
\end{enumerate}

In both cases a contradiction is reached, completing the proof of
the converse.
\end{proof}

\begin{example}\label{weylcaso1}
Let $\K$ be a field and $q\in \K^*$. The \textit{quantum $q$-Weyl algebra}
$A^{q}_{1}(\K)$ is the $\K$-algebra
\[
  A^{q}_{1}(\K) := \K\langle x,y\rangle\big/\langle\,yx - q\,xy - 1\,\rangle.
\]
Setting $q_1=q$ and $q_4=1$, the algebra $A^{q}_{1}(\K)$ is a skew PBW
extension over $\K$ of the form $\sigma(\K)\langle x,y\rangle$.
By Theorem~\ref{PBW2v}, one has that $A^{q}_{1}(\K)$ satisfies the PI property if
and only if $q$ is a root of unity.
\end{example}

\begin{remark}
The algebra $A^{q}_{1}(\K)$ may be regarded as a special case of the
\emph{multiparameter quantum Weyl algebra} $A_{n}^{\mathbf{q},\Lambda}$.
The PI property for this family was studied in Proposition~2 of
\cite{BeraMukherjee2023}; however, the argument used for the converse
implication in that proposition relies on the algebra being a finite
module over its centre, a hypothesis that is not available a priori
for $A^{q}_{1}(\K)$ when $q$ is not a root of unity.
Theorem~\ref{PBW2v} addresses this gap by providing a direct proof
that does not depend on this assumption.
\end{remark}

\begin{example}
Let $\K$ be a field, $q\in \K^*$, and $h\in \K$. The
\textit{algebra of $q$-differential operators} $D_{q,h}[x,y]$ is the
$\K$-algebra generated by $x$ and $y$ subject to the relation
\[
  yx = q\,xy + h.
\]
Setting $q_1=q$ and $q_4=h$, the algebra $D_{q,h}[x,y]$ is a skew PBW extension over $\K$ of the form $\sigma(\K)\langle x,y\rangle$.
By Theorem \ref{PBW2v}, the algebra $D_{q,h}[x,y]$ satisfies the PI property if and only if $q$ is a root of unity.
\end{example}
The following result establishes the transfer of primeness from the base ring to 
the bijective skew PBW extension.

\begin{lemma}[{\cite[Corollary 4.2]{LezamaAcostaReyes2015}}]\label{PrimoPBW}
If $R$ is a prime ring and $A$ is a bijective skew PBW extension of $R$, then  $A$ is prime and $\mathsf{rad}(A) = 0$.
\end{lemma}

Recall that every skew PBW extension 
$A = \sigma(R)\langle x_1,\dots,x_n\rangle$ is a free left module over 
$R$ with basis the standard monomials
\[
\operatorname{Mon}(A) = \bigl\{\, x^{\alpha} = x_1^{\alpha_1} \cdots x_n^{\alpha_n}
\;\bigm|\; \alpha \in \mathbb{N}^n \,\bigr\}.
\]
In particular, when $n \ge 1$, the $R$-linear rank of $A$ is infinite and hence 
$A$ cannot be finitely generated as a module over $R$. This observation is 
central to the proof of the following result.

\begin{theorem}\label{PIprima2}
Let $R$ be a prime $\K$-algebra and let 
$A = \sigma(R)\langle x_1,\dots,x_n\rangle$ be a bijective skew 
PBW extension of $R$, with $n \ge 1$. If $A$ satisfies the PI property, then the center of $A$ 
is non-trivial.
\end{theorem}

\begin{proof}
We proceed by contradiction. Suppose that $Z(A) = \K$. Since $R$ is a prime 
$\K$-algebra, Lemma~\ref{PrimoPBW} implies that $A$ is a prime ring. Applying 
Posner's Theorem (Theorem~\ref{posner}) to $A$ with $S = \K^*$, we obtain 
$Q = AS^{-1} = A$ and $\mathbb{F} = Z(A)S^{-1} = \K$. Consequently, the extension $A$ is a central 
simple $\K$-algebra and, if $d$ is its minimal degree, then
\[
\mathsf{dim}_{\K}(A) = \left(\frac{d}{2}\right)^2 < \infty.
\]
This contradicts the fact that $A$ has infinite $R$-linear rank when $n \ge 1$.
\end{proof}

\begin{corollary}\label{PIprima3}
Let $R$ be a prime $\K$-algebra and let 
$A = \sigma(R)\langle x_1,\dots,x_n\rangle$, with $n \ge 1$, be a bijective skew 
PBW extension of $R$. If $Z(A) = \K$, then $A$ does not satisfy the PI property.
\end{corollary}

\begin{proof}
This is the immediate contrapositive of Theorem~\ref{PIprima2}.
\end{proof}
The natural filtration of a skew PBW extension, first introduced in the proof 
of Theorem~2.2 of \cite{GallegoLezama} and subsequently taken up in 
\cite{LezamaReyes2014}, plays a fundamental role in the study of the PI 
property. In particular, it allows us to reduce the analysis of a general skew 
PBW extension to the quasi-commutative case, which is more tractable. The 
following two results formalize this reduction.

\begin{proposition}[{\cite[Theorem 2.2]{LezamaReyes2014}}]\label{filtracion}
Let $A = \sigma(R)\langle x_1,\dots,x_n\rangle$ be a skew PBW extension. For 
each $m \in \mathbb{Z}$, define
\[
F_m :=
\begin{cases}
0, & \text{if } m \leq -1,\\[1mm]
R, & \text{if } m = 0,\\[1mm]
\{\, f \in A \mid \deg(f) \leq m \,\}, & \text{if } m \geq 1.
\end{cases}
\]
Then $\{F_m\}_{m \in \mathbb{Z}}$ is a filtration of $A$ and the associated 
graded ring
\[
\mathsf{Gr}(A) := \bigoplus_{m \geq 0} F_m/F_{m-1}
\]
is a quasi-commutative skew PBW extension of $R$. In particular, if $A$ is 
bijective, then $\mathsf{Gr}(A)$ is bijective as well.
\end{proposition}

\begin{lemma}[{\cite[Proposition 1.2]{LeroyMatczuk2007}}]\label{graduadoasociado}
Let $A = \bigcup_{i \geq 0} A_i$ be a filtered ring with associated graded ring
\[
\mathsf{gr}(A) := \bigoplus_{i \geq 0} A_i/A_{i-1}.
\]
If $A$ satisfies a polynomial identity, then $\mathsf{gr}(A)$ satisfies a 
polynomial identity as well.
\end{lemma}

Combining the two results above yields necessary conditions for a skew PBW 
extension over a field to satisfy the PI property.

\begin{theorem}\label{necesariaPI}
Let $A = \sigma(\K)\langle x_1,\dots,x_n\rangle$ be a skew PBW extension over a 
field $\K$. If $A$ satisfies the PI property, then, for all $1 \leq i < j \leq n$, 
the scalar $c_{i,j} \in \K^{*}$ appearing in the relation
\[
x_j x_i - c_{i,j}\, x_i x_j \in \K + \K x_1 + \cdots + \K x_n
\]
is a root of unity.
\end{theorem}

\begin{proof}
By Proposition~\ref{filtracion}, the extension $A$ is a filtered ring whose associated graded 
ring $\mathsf{Gr}(A)$ is a quasi-commutative skew PBW extension of $\K$. 
Since $A$ is defined over the field $\K$, every element of $\K$ is constant in $A$ 
and this property is inherited by $\mathsf{Gr}(A)$; in particular, 
$\mathsf{Gr}(A)$ is constant and hence semicommutative. Moreover, the 
scalars $c_{i,j}$ defined in $A$ are preserved in $\mathsf{Gr}(A)$. Since 
$A$ satisfies the PI property, Lemma~\ref{graduadoasociado} implies that 
$\mathsf{Gr}(A)$ does as well. Fix $1 \le i < j \le n$ and consider the 
subalgebra
\[
B_{ij} := \K\langle \overline{x}_i, \overline{x}_j \rangle \subseteq 
\mathsf{Gr}(A),
\]
which satisfies
\[
B_{ij} \cong \K\langle u,v\rangle/(vu - c_{i,j}uv) = \mathcal{O}_{c_{i,j}}(\K^2).
\]
Since $\mathsf{Gr}(A)$ satisfies the PI property, the subalgebra $B_{ij}$ 
does as well. By \cite[\S 7.1]{DeConciniProcesi1993}, the algebra $\mathcal{O}_{c_{i,j}}(\K^2)$ 
satisfies the PI property if and only if $c_{i,j}$ is a root of unity in $\K$.
\end{proof}

\begin{remark}
The converse of Theorem~\ref{necesariaPI} does not hold in general: a skew PBW 
extension may have trivial center and, even when all parameters $c_{i,j}$ are 
roots of unity, fail to satisfy the PI property.
\end{remark}
\begin{theorem}\label{necesariaPI2}
Let $A = \sigma(R)\langle x_1,\dots,x_n\rangle$ be a constant skew PBW extension 
of an integral domain $R$. If $A$ satisfies the PI property, then, for all 
$1 \leq i < j \leq n$, the scalar $c_{i,j} \in R$ appearing in the relation
\[
x_j x_i - c_{i,j}\, x_i x_j \in R + Rx_1 + \cdots + Rx_n
\]
is a root of unity.
\end{theorem}

\begin{proof}
By Proposition \ref{filtracion}, the extension $A$ is a filtered ring whose associated graded 
ring $\mathsf{Gr}(A)$ is a semicommutative skew PBW extension of the 
integral domain $R$, in which the scalars $c_{i,j}$ defined in $A$ are preserved. 
Since $A$ satisfies the PI property, Lemma \ref{graduadoasociado} implies that 
$\mathsf{Gr}(A)$ does as well; hence Theorem \ref{PI-semi-comm} yields that 
each $c_{i,j}$ is a root of unity.
\end{proof}

\begin{example}
Let $A = \sigma(\K)\langle x_1,\dots,x_n\rangle$ be a skew PBW extension over a 
field $\K$ defined by
\[
x_j x_i = q_{ij}\, x_i x_j + a^{(1)}_{ij}\, x_1 + \cdots + a^{(n)}_{ij}\, x_n 
+ d_{ij}, \qquad 1 \le i < j \le n,
\]
where $q_{ij} \in \K^{*}$, $a^{(t)}_{ij}, d_{ij} \in \K$ for $1 \le t \le n$, and
\[
q_{ij} = \frac{l_{ij}}{k_{ij}}, \text{ with }
\mathsf{gcd}(l_{ij}, l_{rs}) = \mathsf{gcd}(k_{ij}, k_{rs}) = \mathsf{gcd}(l_{ij}, k_{rs}) = 1
\]
for all $i,j,r,s \in \{1,2,\dots,n\}$. If $A$ satisfies the PI property, then 
Theorem \ref{necesariaPI} implies that each $q_{ij}$ is a root of unity. This is 
consistent with Theorem 2.5 of \cite{LezamaVenegas2020}, which shows that if some 
$q_{ij}$ is not a root of unity, then $Z(A) = \K$; in that case, since $A$ is a 
bijective skew PBW extension over the field $\K$, Lemma \ref{PrimoPBW} implies 
that $A$ is prime and Theorem~\ref{PIprima2} then yields that its center cannot 
be trivial if $A$ satisfies the PI property.
\end{example}

The results established so far provide necessary conditions for the PI property, 
but are not sufficient in general. However, for skew PBW extensions whose center 
is known, these criteria allow one to determine precisely whether the PI property 
holds. In \cite{LezamaVenegas2020}, Lezama and Venegas collect and study the 
center of various skew PBW extensions; the results below illustrate how this 
information can be exploited. The following theorem, taken from 
\cite{LezamaVenegas2020}, describes a broad family of algebras whose center had 
not been computed in the literature prior to that work.

\begin{theorem}[{\cite[Theorem 2.6]{LezamaVenegas2020}}]\label{PI3}
Let $A = \sigma\!\bigl(\K[x_1,\dots,x_n]\bigr)\langle y_1,\dots,y_n\rangle$ be 
the skew PBW extension defined by the relations 
\begin{align*}
y_i y_j &= y_j y_i, \qquad 1 \le i < j \le n, \\[2pt]
y_j x_i &= x_i y_j, \qquad i \ne j, \\[2pt]
y_i x_i &= q_i\, x_i y_i + d_i\, y_i + a_i, \qquad 1 \le i \le n.
\end{align*}
Then,
\begin{enumerate}[\rm (i)]
\item $Z(A) = \K$ in any of the following cases:
  \begin{enumerate}
  \item[\rm (ia)] $q_i$ is not a root of unity, for all $1 \le i \le n$.
  \item[\rm (ib)] There exists $i$ such that $q_i = 1$ and $d_i \neq 0$ or 
  $a_i \neq 0$.
  \end{enumerate}
\item If $q_i$ is a primitive root of unity of order $l_i \ge 2$, then $y_i^{\,l_i}$ is central, for all 
$1 \le i \le n$, and 
$\K \subsetneq Z(A)$.
\end{enumerate}
\end{theorem}

The theorem above yields precise criteria under which these algebras satisfy the 
PI property.

\begin{theorem}\label{PI5}
Let $\K$ be a field and let 
$A = \sigma\!\bigl(\K[x_1,\dots,x_n]\bigr)\langle y_1,\dots,y_n\rangle$ be the 
skew PBW extension defined as in Theorem \ref{PI3}.
If $q_i$ is a primitive root of unity of order $l_i \ge 2$, for all 
$1 \le i \le n$, then $A$ satisfies the PI property.
\end{theorem}

\begin{proof}
By Theorem~\ref{PI3}-(ii), the power $y_i^{\,l_i} \in Z(A)$ for each $1 \le i \le n$, so 
that $Z_0 := \K[y_1^{l_1}, \dots, y_n^{l_n}] \subseteq Z(A)$. Every element of 
$A$ can be written as a $Z_0$-linear combination of monomials in the finite set
\[
\bigl\{\, y_1^{m_1} \cdots y_n^{m_n} \mid 0 \le m_i < l_i,\ 1 \le i \le n \,\bigr\}.
\]
Therefore $A$ is finitely generated as a $Z_0$-module, for \cite[\S 13.1.13, Corollary]{mcconnell}, the extension $A$ satisfies the PI property.
\end{proof}

\begin{theorem}\label{PI9}
Under the same conditions as in the previous theorem, if $A$ satisfies the PI property, then the following conditions hold:
\begin{enumerate}
  \item[\rm (ia)] There exists $i \in \{1,\dots,n\}$ such that $q_i$ is a root of 
  unity.
  \item[\rm (ib)] For all $i \in \{1,\dots,n\}$, either $q_i \neq 1$ or 
  $(d_i = 0$ and $a_i = 0)$.
\end{enumerate}
\end{theorem}

\begin{proof}
Suppose that at least one of conditions (ia) or (ib) fails. Then 
Theorem \ref{PI3}-(i) implies $Z(A) = \K$. Since $\K[x_1,\dots,x_n]$ is an 
integral domain and hence a prime ring, and $A$ is a bijective skew PBW extension 
of $\K[x_1,\dots,x_n]$, Lemma \ref{PrimoPBW} implies that $A$ is prime. Applying 
Corollary \ref{PIprima3}, we conclude that $A$ does not satisfy the PI property, 
contradicting the hypothesis. Therefore, conditions (ia) and (ib) must hold 
simultaneously.
\end{proof}

Combining Lemma \ref{PrimoPBW}, Theorem \ref{PI3}-(i), and 
Corollary \ref{PIprima3}, the bijective skew PBW extensions presented below are 
prime, have trivial center, and consequently do not satisfy the PI property.
\begin{example}
\begin{enumerate}[(1)]
  \item \emph{Algebra of linear partial shift operators.}
  The $\K$-algebra of linear partial shift operators with polynomial coefficients denoted by $\K[t_1,\ldots,t_n][E_1,\ldots,E_n]$
  (respectively, with rational coefficients  
  $\K(t_1,\ldots,t_n)[E_1,\ldots,E_n]$), is subject to the relations
  \begin{align*}
    t_j t_i &= t_i t_j, && 1 \le i < j \le n,\\
    E_j E_i &= E_i E_j, && 1 \le i < j \le n,\\
    E_i t_i &= (t_i+1)\,E_i = t_i E_i + E_i, && 1 \le i \le n,\\
    E_j t_i &= t_i E_j, && i \ne j.
  \end{align*}
  In particular, $q_i = 1$, $d_i = 1$, and $a_i = 0$ for all $1 \le i \le n$. 
  This algebra is a bijective skew PBW extension of $\K[t_1,\ldots,t_n]$ 
  (respectively, of $\K(t_1,\ldots,t_n)$).

  \item \emph{Algebra of linear partial difference operators.}
  The $\K$-algebra of linear partial difference operators with polynomial 
  coefficients  
  $\K[t_1,\ldots,t_n][\Delta_1,\ldots,\Delta_n]$ (respectively, with rational coefficients), is subject to the relations
  \begin{align*}
    t_j t_i &= t_i t_j, && 1 \le i < j \le n,\\
    \Delta_i t_i &= (t_i+1)\,\Delta_i = t_i \Delta_i + \Delta_i + 1, 
    && 1 \le i \le n,\\
    \Delta_i t_j &= t_j \Delta_i, && i \ne j,\\
    \Delta_j \Delta_i &= \Delta_i \Delta_j, && 1 \le i < j \le n.
  \end{align*}
  In particular, $q_i = 1$, $d_i = 1$, and $a_i = 1$ for all $1 \le i \le n$. 
  This algebra is a bijective skew PBW extension of $\K[t_1,\ldots,t_n]$.

  \item \emph{Additive analogue of the Weyl algebra.}
  This $\K$-algebra, denoted $A_n(q_1,\dots,q_n)$, is generated by the variables 
  $x_1,\dots,x_n, y_1,\dots,y_n$ subject to the relations
  \begin{align*}
    x_j x_i &= x_i x_j, \qquad y_j y_i = y_i y_j, && 1 \le i,j \le n,\\[2pt]
    y_j x_i &= x_i y_j, && i \ne j,\\[2pt]
    y_i x_i &= q_i\, x_i y_i + 1, && 1 \le i \le n,
  \end{align*}
  where each $q_i \in \K^*$ is not a root of unity. One has 
  $A_n(q_1,\dots,q_n) \cong \sigma\!\bigl(\K[x_1,\dots,x_n]\bigr)
  \langle y_1,\dots,y_n\rangle$, with $d_i = 0$ and $a_i = 1$, for 
  $1 \le i \le n$. Note that if each $q_i$ were a non-trivial root of unity, then this algebra $A_n(q_1,\dots,q_n)$ would satisfy the PI property.
\end{enumerate}
\end{example}

The following examples illustrate algebras for which the PI property follows 
from Theorem~\ref{PI5}, under the assumption that $q_i$ is a primitive root of 
unity of order $l_i \ge 2$, for each $1 \le i \le n$.

\begin{example}
\begin{enumerate}[(1)]
  \item \emph{Algebra of $q$-differential operators.}
  The algebra of $q$-differential operators $D_{q,h}[x,y]$ is the $\K$-algebra 
  generated by $x$ and $y$ subject to the relation $yx = q\,xy + h$, where 
  $q, h \in \K^*$. Here $q_1 = q$, $d_1 = 0$, and $a_1 = h$.


  \item \emph{Algebra of linear partial $q$-dilation operators.}
  The algebra $\K[t_1,\ldots,t_n]\langle H_1,\ldots,H_n\rangle$ is generated by 
  $t_1,\ldots,t_n, H_1,\ldots,H_n$ subject to the relations
  \begin{align*}
    t_j t_i &= t_i t_j, \qquad H_j H_i = H_i H_j, && 1 \le i < j \le n,\\
    H_i t_i &= q\, t_i H_i, && 1 \le i \le n,\\
    H_j t_i &= t_i H_j, && i \ne j,
  \end{align*}
  where $q \in \K^*$. Here $q_i = q$, $d_i = 0$, and $a_i = 0$ for all 
  $1 \le i \le n$.

  \item \emph{Algebra of linear partial $q$-differential operators.}
  The algebra $\K[t_1,\ldots,t_n]\langle D_1,\ldots,D_n\rangle$ is generated by 
  $t_1,\ldots,t_n, D_1,\ldots,D_n$ subject to the relations
  \begin{align*}
    t_j t_i &= t_i t_j, \qquad D_j D_i = D_i D_j, && 1 \le i < j \le n,\\
    D_i t_i &= q\, t_i D_i + 1, && 1 \le i \le n,\\
    D_j t_i &= t_i D_j, && i \ne j,
  \end{align*}
  where $q \in \K^*$. Here $q_i = q$, $d_i = 0$, and $a_i = 1$, for all 
  $1 \le i \le n$.
\end{enumerate}
\end{example}
\subsection{The PI property of other skew PBW extensions}
\noindent
In this subsection we study the PI property of further algebras whose center has 
already been determined in the literature and which admit an interpretation as 
skew PBW extensions. Unless otherwise stated, the base field $\K$ is assumed to 
have characteristic zero.

\subsection*{Extended Weyl algebra}

The extended Weyl algebra is defined as
\[
B_n(\K) := \K(t_1,\dots,t_n)[x_1;\,\partial/\partial t_1]
\cdots[x_n;\,\partial/\partial t_n],
\]
where $\K(t_1,\dots,t_n)$ denotes the field of fractions of $\K[t_1,\dots,t_n]$. This algebra can be seen as a skew PBW extension, where the endomorphisms in  Proposition \ref{prop23} satisfy
$\sigma_i = \mathsf{id}$, for each $1 \le i \le n$; consequently,
\[
B_n(\K) \cong \sigma\bigl(\K(t_1,\dots,t_n)\bigr)\langle x_1, \dots, x_n \rangle.
\]

\begin{proposition}
The extended Weyl algebra $B_n(\K)$ does not satisfy the PI property.
\end{proposition}

\begin{proof}
It is known that $Z(B_n(\K)) = \K$, (see e.g. \cite[\S 1.3.9]{mcconnell}, \cite[Exercise 6H]{GoodearlWarfield2004} or \cite{Dixmier1996}). Since 
$\K(t_1,\dots,t_n)$ is a field and hence a prime ring, Lemma \ref{PrimoPBW} 
implies that $B_n(\K)$ is a prime skew PBW extension. By 
Corollary \ref{PIprima3}, the algebra $B_n(\K)$ does not satisfy the PI property.
\end{proof}

\subsection*{Shift operator algebra $S_h$}

Let $h \in \K$. The bijective-type Ore extension 
$S_h := \K[t][x_h;\sigma_h,\delta_h]$, with $\sigma_h(p(t)) := p(t-h)$ and 
$\delta_h := 0$, is a skew PBW extension of $\K[t]$:
\[
S_h \cong \sigma(\K[t])\langle x_h \rangle.
\]

\begin{proposition}
The shift operator algebra $S_h$ does not satisfy the PI property.
\end{proposition}

\begin{proof}
By Example~1.9-(b) of \cite{LezamaVenegas2020} one has that $Z(S_h) = \K$. Since $\K[t]$ is 
an integral domain and hence a prime ring, Lemma \ref{PrimoPBW} implies that 
$S_h$ is a prime skew PBW extension. By Corollary~\ref{PIprima3}, the algebra $S_h$ does 
not satisfy the PI property.
\end{proof}

\subsection*{Quantum enveloping algebra $U_q(\mathfrak{sl}(2,\K))$}

The algebra $U_q(\mathfrak{sl}(2,\K))$ is defined as the $\K$-algebra generated 
by $x, y, z, z^{-1}$ subject to the relations
\[
zz^{-1} = z^{-1}z = 1, \qquad xz = q^{-2}zx, \qquad yz = q^{2}zy, \qquad
xy - yx = \frac{z - z^{-1}}{q - q^{-1}},
\]
where $q \neq \pm 1$. These relations show that 
$U_q(\mathfrak{sl}(2,\K)) = \sigma\bigl(\K[z,z^{-1}]\bigr)\langle x,y\rangle$. 
For arbitrary $q$, the center of this algebra is generated by the Casimir element 
$C_2 = (q^{2}-1)^{2}xy + qz + q^{3}z^{-1}$; if $q$ is an $n$-th root of unity 
with $n \ge 2$, in \cite[Theorem 4.2]{DeConciniKac1990} was proved that
\[
Z\bigl(U_q(\mathfrak{sl}(2,\K))\bigr) = \K\bigl[C_2, x^{n}, y^{n}, z^{n}, 
z^{-n}\bigr].
\]

\begin{theorem}\label{PI-Uqsl2}
Let $q \in \K^{*}$ with $q \neq \pm 1$. Then $U_q(\mathfrak{sl}(2,\K))$ satisfies 
the PI property if and only if $q$ is a root of unity.
\end{theorem}

\begin{proof}
Suppose that $U_q(\mathfrak{sl}(2,\K))$ satisfies the PI property and consider 
the subalgebra
\[
B := \K\langle y, z, z^{-1} \mid zz^{-1} = z^{-1}z = 1,\ yz = q^{2}zy\rangle.
\]
Then $B \cong \K[z,z^{-1}][y;\sigma]$ is an Ore extension of endomorphism type 
over the prime ring $\K[z,z^{-1}]$, where $\sigma(z) = q^{2}z$. Since $B$ is a 
subalgebra of $U_q(\mathfrak{sl}(2,\K))$ and the latter satisfies the PI 
property, $B$ does as well. By Proposition 2.5 of \cite{LeroyMatczuk2007}, the 
restriction $\sigma|_{Z(\K[z,z^{-1}])}$ has finite order. Since $\K[z,z^{-1}]$ is 
commutative, $Z(\K[z,z^{-1}]) = \K[z,z^{-1}]$ and hence there exists $n \ge 1$ 
such that $\sigma^n = \mathrm{id}_{\K[z,z^{-1}]}$; in particular, 
$\sigma^n(z) = q^{2n}z = z$, so $q^{2n} = 1$ and $q$ is a root of unity.

Conversely, suppose that $q$ is an $n$-th root of unity with $n \ge 2$. Since 
$$Z\bigl(U_q(\mathfrak{sl}(2,\K))\bigr) = \K\bigl[C_2, x^{n}, y^{n}, z^{n}, 
z^{-n}\bigr],$$ the algebra $U_q(\mathfrak{sl}(2,\K))$ is finitely generated as a 
$Z\bigl(U_q(\mathfrak{sl}(2, \K))\bigr)$-module and, by \cite[\S 13.1.13, Corollary]{mcconnell},  $U_q(\mathfrak{sl}(2,\K))$ satisfies the PI property.
\end{proof}

\subsection*{Quantum symplectic space $\mathcal{O}_q(\mathfrak{sp}(\K^{2n}))$}

The quantum symplectic space $\mathcal{O}_q(\mathfrak{sp}(\K^{2n}))$ is the 
algebra generated by $x_1,\dots,x_{2n}$ subject to the relations
\begin{align}
x_j x_i = q\,x_i x_j, \qquad i < j, \quad i' \ne j,
\end{align}
\begin{align}\label{identity2}
x_{i'}x_i = q^{2} x_i x_{i'} + (q^{2}-1)\sum_{1 \le k < i} q^{\,i-k} 
x_k x_{k'}, \qquad i < i',
\end{align}
where $i' = 2n - i + 1$, for $1 \le i \le 2n$. The previous defining relations allow to prove that 
$\mathcal{O}_q(\mathfrak{sp}(\K^{2n}))$ is isomorphic to a bijective skew PBW 
extension of the form
\[
\mathcal{O}_q(\mathfrak{sp}(\K^{2n}))
\cong
\sigma\bigl(\cdots\sigma(\sigma(\K)\langle x_1,x_{2n}\rangle)
\langle x_2,x_{2n-1}\rangle)\cdots\bigr)\langle x_n,x_{n+1}\rangle.
\]
In his article \cite{Zhang2000}, Zhang studied the irreducible representations of 
$\mathcal{O}_q(\mathfrak{sp}(\K^{2n}))$. Among other relevant results, he proved that if 
$m$ is a positive odd integer and $q$ is a primitive $m$-th root of unity, then 
$Z\bigl(\mathcal{O}_q(\mathfrak{sp}(\K^{2n}))\bigr) = \K[x_1^{m},\dots,
x_{2n}^{m}]$, see \cite[Theorem 4.3]{Zhang2000}. Moreover, if $q$ is an 
$m$-th root of unity (not necessarily primitive), then each $x_i^{m}$ is a 
central element of $\mathcal{O}_q(\mathfrak{sp}(\K^{2n}))$, see 
\cite[Lemma 4.2]{Zhang2000}. If $q$ is not a root of unity, the center of this algebra is trivial, as proven below.

\begin{proposition}\label{centerno-root}
Let $\K$ be a field and $q \in \K^{*}$ an element that is not a root of unity. 
Then the center of $\mathcal{O}_q\bigl(\mathfrak{sp}(\K^{2n})\bigr)$ is trivial.
\end{proposition}

\begin{proof}
Take $A:=\mathcal{O}_q\bigl(\mathfrak{sp}(\K^{2n})\bigr)$ and, for $1\leq i\leq n$, define $\omega(i):=i$ while $\omega(i'):=1$, and extend this weight $\omega$ additively to standard monomials as follows:
\[\omega(x_1^{\alpha_1}\cdots x_{2n}^{\alpha_{2n}}):=\sum_{i=1}^{2n}\alpha_{i}\omega(x_i).\]
Endow this algebra with the weighted filtration $\{F^{\omega}_{m}(A)\}_{m\geq 0}$ determined by $\omega$: 
\[ F^{\omega}_m(A):= {}_{\K} \left\langle x_1^{\alpha_1}\cdots x_{2n}^{\alpha_{2n}} \mid \omega(x_1^{\alpha_1}\cdots x_{2n}^{\alpha_{2n}})\leq m \right\rangle. \] 
In particular, note that in the relation (\ref{identity2}) the weighted graded of $x_kx_{k'}$, for $1\leq k<i$, is strictly less than that of $x_ix_{i'}$, so
the associated graded algebra \(\mathsf{gr}^{\omega}(A)\) is generated by the principal symbols $\overline{x}_1,\ldots,\overline{x}_{2n}$,  and its defining relations are obtained by retaining the weighted highest-degree terms of the defining relations of \(A\). Hence 
\begin{align}\label{relationsgraded}
\overline{x}_j\overline{x}_i& = q\,\overline{x}_i\overline{x}_j, \qquad i<j,\quad j\neq i', \\
\overline{x}_{i'}\overline{x}_i &= q^2\,\overline{x}_i\overline{x}_{i'}, \qquad i<i',\notag
\end{align}
with  $i' = 2n - i + 1$, for $1 \le i \le 2n$, see \cite[\S 4]{Zhang2000} where this algebra is called the \textit{associated quasipolynomial algebra} of $\mathcal{O}_q(\mathfrak{sp}(\K^{2n}))$, and is denoted by $\overline{\mathcal{O}_q(\mathfrak{sp}(\K^{2n}))}$. Thus \(\mathsf{gr}^{\omega}(A)\) is a quantum affine space. We first prove that  $Z(\mathsf{gr}^{\omega}(A))=\K$.  For \(\alpha=(\alpha_1,\ldots,\alpha_{2n})\in \mathbb{N}^{2n}\), write 

\[ \overline{x}^{\alpha} := \overline{x}_1^{\alpha_1}\cdots \overline{x}_{2n}^{\alpha_{2n}},\]
and, for \(1\leq i<j\leq 2n\), define 
\[ \varepsilon_{ij}:= \begin{cases} 2, & j=i',\\ 1, & j\neq i'. \end{cases} \] 
Thus  $\overline{x}_j\overline{x}_i = q^{\varepsilon_{ij}} \overline{x}_i\overline{x}_j$, for $i<j$.  The latter allows us to introduce a skew-symmetric integer matrix 
\[ H=(h_{rs})_{1\leq r,s\leq 2n}\]
such that 
$\overline{x}_r\overline{x}_s = q^{h_{rs}} \overline{x}_s\overline{x}_r$.
Explicitly, for \(r>s\), 

\[ h_{rs}:= \begin{cases} 2, & r=s',\\ 1, & r\neq s', \end{cases} \qquad h_{sr}:=-h_{rs}, \qquad h_{rr}:=0. \] 
We claim that, in order to determine the center of \(\mathsf{gr}^{\omega}(A)\), it is enough to determine its central standard  monomials. Indeed, let $f=\sum_{\alpha} c_\alpha \overline{x}^{\alpha}$, with  $c_{\alpha}\in \K$, be a non-zero central element of \(\mathsf{gr}^{\omega}(A)\). Fix \(\ell\in\{1,\ldots,2n\}\). From (\ref{relationsgraded}) it follows that
\[ \overline{x}_{\ell}\overline{x}^{\alpha} = q^{\lambda_{\ell}(\alpha)}\overline{x}^{\alpha}\overline{x}_{\ell}, \] 
where 
$\lambda_{\ell}(\alpha):=\sum_{s=1}^{2n} h_{\ell s}\alpha_s$.  
Since \(f\) is central, one has \(\overline{x}_{\ell} f=f\overline{x}_{\ell}\). Therefore, 
\[ \sum_{\alpha} c_{\alpha} q^{\lambda_{\ell}(\alpha)} \overline{x}^{\alpha}\overline{x}_{\ell} = \sum_{\alpha} c_{\alpha} \overline{x}^{\alpha}\overline{x}_{\ell}. \] 
Equivalently, 
\[ \sum_{\alpha} c_{\alpha}\bigl(q^{\lambda_{\ell}(\alpha)}-1\bigr) \overline{x}^{\alpha}\overline{x}_{\ell} = 0. \] 
Note that, for each \( \alpha\), the product \(\overline{x}^{\alpha}\overline{x}_{\ell}\) is a non-zero scalar multiple of the standard monomial \(\overline{x}^{\alpha+\mathbf{e}_{\ell}}\), where \(\mathbf{e}_{\ell}\) is the \(\ell\)-th vector of the canonical basis of \(\mathbb{N}^{2n}\). Since the monomials  $\overline{x}^{\alpha+\mathbf{e}_{\ell}}$  are distinct and form part of the $\K$-basis for $\mathsf{gr}^{\omega}(A)$, they are \(\K\)-linearly independent and  $c_{\alpha}\bigl(q^{\lambda_{\ell}(\alpha)}-1\bigr)=0$, for every non-zero $c_{\alpha}$ in the presentation of $f$. Thus $q^{\lambda_{\ell}(\alpha)}=1$ for each of these coefficients. As \(q\) is not a root of unity, it follows that $\lambda_{\ell}(\alpha)=0$. This holds for every \(1\leq \ell\leq 2n\). Hence, each monomial \(\overline{x}^{\alpha}\) appearing in the expansion of \(f\) commutes with every generator \(\overline{x}_{\ell}\). Thus each such monomial is central. Consequently, every central element of \(\mathsf{gr}^{\omega}(A)\) is a \(\K\)-linear combination of central monomials. Therefore, in this quantum affine space, the center is completely determined by its central monomials. 

On the other hand, a monomial \(\overline{x}^{\alpha}\) is central if and only if it commutes with each generator \(\overline{x}_{\ell}\). This condition is equivalent to 
\[ q^{(H\alpha)_{\ell}}=1, \quad \text{for all } \ell=1,\ldots,2n. \]
Since \(q\) is not a root of unity, this is equivalent to $H\alpha=\mathbf{0}$. 
We show that this latter implies \(\alpha=\mathbf{0}\). Define $T:=\alpha_1+\cdots+\alpha_{2n}$.  The first row of \(H\alpha=\mathbf{0}\) gives \[ -\alpha_2-\alpha_3-\cdots-\alpha_{2n-1}-2\alpha_{2n}=0. \] Equivalently, $T=\alpha_1-\alpha_{2n}$. Similarly, the last row of $H\alpha=\mathbf{0}$ gives \[ 2\alpha_1+\alpha_2+\cdots+\alpha_{2n-1}=0; \] that is, $T= \alpha_{2n}-\alpha_1$. Therefore, $T=-T$
and hence $T=\mathbf{0}$. Since each $\alpha_i$ is a natural number, $T=\mathbf{0}$ implies $\alpha_1=\cdots=\alpha_{2n}=0.$ Thus the only central monomial of $\mathsf{gr}^{\omega}(A)$ is $1$. Consequently, $Z(\mathsf{gr}^{\omega}(A))=\K$. Now, let \(z\in Z(A)\) be a non-zero central element of $A$. If \(z\notin \K\), take \(d\geq 1\) minimal such that $z\in F^{\omega}_d(A)$.  Let \[ \overline{z}:=z+F^{\omega}_{d-1}(A)\in F^{\omega}_d(A)/F^{\omega}_{d-1}(A) \] be the principal symbol of $z$ in $\mathsf{gr}(A)$. Since $z$ commutes with every generator $x_i$, its principal symbol $\overline{z}$ commutes with every $\overline{x}_i$. In consequence, one obtains $\overline{z}\in Z(\mathsf{gr}^{\omega}(A))$. But \(Z(\mathsf{gr}^{\omega}(A))=\K\), and every non-zero element of \(\K\) has degree \(0\), contradicting \(d\geq 1\). Therefore \(z\in \K\) and $Z(A)=\K$.
\end{proof}
\begin{remark} An alternative proof of Proposition \ref{centerno-root} using primitive ideals can be found in \cite[Corollary 7.2]{Oh1995}.
\end{remark}

\begin{theorem}\label{PI-O_qSP(K2n)}
$\mathcal{O}_q(\mathfrak{sp}(\K^{2n}))$ satisfies the PI property if and only if 
$q$ is a root of unity.
\end{theorem}

\begin{proof}
Suppose that $\mathcal{O}_q(\mathfrak{sp}(\K^{2n}))$ satisfies the PI property 
and that $q$ is not a root of unity. Applying Lemma \ref{PrimoPBW} repeatedly, one obtains that
$\mathcal{O}_q(\mathfrak{sp}(\K^{2n}))$ is a prime ring; by the proposition 
above, its center is trivial, contradicting Corollary \ref{PIprima3}.

Conversely, suppose that $q$ is an $m$-th root of unity. By 
\cite[Lemma 4.2]{Zhang2000}, the power $x_i^{m}$ is central in 
$\mathcal{O}_q(\mathfrak{sp}(\K^{2n}))$, for all $i = 1,\dots,2n$, so
\[
Z_0 := \K[x_1^{m},\dots,x_{2n}^{m}] \subseteq 
Z\bigl(\mathcal{O}_q(\mathfrak{sp}(\K^{2n}))\bigr).
\]
Therefore $\mathcal{O}_q(\mathfrak{sp}(\K^{2n}))$ is finitely generated as a 
$Z_0$-module and, by \cite[\S 13.1.13, Corollary]{mcconnell}, satisfies the PI 
property.
\end{proof}

\subsection*{Mixed algebra $D_h$}

Let $h \in \K$. The bijective-type Ore extension
\[
D_h := \K[t]\bigl[x;\,\mathrm{id}_{\K[t]},\,\tfrac{d}{dt}\bigr]
\bigl[x_h;\,\sigma_h,\,\delta_h\bigr],
\]
where $\sigma_h(x) := x$, is a skew PBW extension of $\K[t]$. Specifically,
$D_h \cong \sigma(\K[t])\langle x, x_h\rangle$.

\begin{proposition}
The mixed algebra $D_h$ does not satisfy the PI property.
\end{proposition}

\begin{proof}
By Example~1.9-(c) of \cite{LezamaVenegas2020}, one has that $Z(D_h) = \K$. Since $\K[t]$ is  an integral domain and hence a prime ring, Lemma \ref{PrimoPBW} implies that 
$D_h$ is a prime skew PBW extension. By Corollary \ref{PIprima3}, this algebra does 
not satisfy the PI property.
\end{proof}

\subsection*{$q$-Heisenberg algebras $H_n(q)$}
\noindent
Let $\K$ be a field and $q \in \K^{*}$. The $q$-Heisenberg algebra $H_n(q)$ is 
the $\K$-algebra generated by $$x_1,\dots,x_n,\; y_1,\dots,y_n,\; z_1,\dots,z_n$$ 
subject to the relations
\[
\renewcommand{\arraystretch}{1.15}
\begin{array}{@{}l@{\qquad}l@{\qquad}l@{}}
x_j x_i = x_i x_j, & y_j y_i = y_i y_j, & z_j z_i = z_i z_j,\\
y_j x_i = x_i y_j, & z_i y_j = y_j z_i, & z_j x_i = x_i z_j,\\
y_i x_i = q\,x_i y_i, & z_i y_i = q\,y_i z_i, &
z_i x_i = q^{-1} x_i z_i + y_i,
\end{array}
\qquad
\begin{array}{l}
i \ne j,\\
1 \le i < j \le n.
\end{array}
\]
It is not difficult to show that  $H_n(q) \cong \sigma\!\bigl(\K[y_1,\dots,y_n]\bigr)\langle 
x_1,\dots,x_n;\, z_1,\dots,z_n\rangle$.

\begin{theorem}\label{PI-Hnq}
Let $q \in \K^{*}$. The algebra $H_n(q)$ satisfies the PI property if and only 
if $q$ is a root of unity.
\end{theorem}

\begin{proof}
Suppose that $q$ is a root of unity and let $l = \mathsf{ord}(q)$. By 
Theorem 3.5 of \cite{LezamaVenegas2020}, the elements $x_i^l$, $y_i^l$, and 
$z_i^l$ are central in $H_n(q)$, so
\[
Z_0 := \K\bigl[x_1^l,\dots,x_n^l,\, y_1^l,\dots,y_n^l,\, 
z_1^l,\dots,z_n^l\bigr] \subseteq Z(H_n(q)).
\]
Every element of $H_n(q)$ can be written as a $Z_0$-linear combination of 
monomials in the finite set
\[
\bigl\{\, x_1^{a_1}\cdots x_n^{a_n}\, y_1^{b_1}\cdots y_n^{b_n}\, 
z_1^{c_1}\cdots z_n^{c_n} \;\bigm|\; 0 \le a_i, b_i, c_i < l \,\bigr\}.
\]
Hence $H_n(q)$ is finitely generated as a $Z_0$-module and, by 
 \cite[\S 13.1.13, Corollary]{mcconnell}, this algebra satisfies the PI property.

Conversely, suppose that $H_n(q)$ satisfies the PI property. For each 
$i \in \{1,\dots,n\}$, the subalgebra $A_i := \K\langle x_i, y_i\rangle \subset 
H_n(q)$ is isomorphic to the quantum plane $\mathcal{O}_q(\K^2)$. Since $H_n(q)$ 
is a PI ring, $A_i$ is as well and, by \cite[\S~7.1]{DeConciniProcesi1993}, the parameter $q$ 
is a root of unity.
\end{proof}

\subsection*{Woronowicz algebra $W_\nu(\mathfrak{sl}(2,\K))$}
\noindent
This algebra was introduced by Woronowicz in \cite{Woronowicz1987} and is 
generated by $x, y, z$ subject to the relations
\[
xz - \nu^{4}zx = (1+\nu^{2})x, \qquad
xy - \nu^{2}yx = \nu z, \qquad
zy - \nu^{4}yz = (1+\nu^{2})y,
\]
where $\nu \in \K^{*}$ is not a root of unity. One has 
$W_\nu\!\bigl(\mathfrak{sl}(2,\K)\bigr)$ is isomorphic to a bijective skew PBW extension of $\K$ in the variables $x$, $y$ and $z$.
\begin{proposition}
The Woronowicz algebra $W_\nu(\mathfrak{sl}(2,\K))$ does not satisfy the PI 
property.
\end{proposition}

\begin{proof}
Since $W_\nu(\mathfrak{sl}(2,\K)) \cong \sigma(\K)\langle x,y,z\rangle$ is a 
bijective skew PBW extension over the field $\K$, Lemma~\ref{PrimoPBW} implies 
it is a prime ring. By Theorem~3.2 of \cite{LezamaVenegas2020}, the center of
$W_\nu(\mathfrak{sl}(2,\K))$ is trivial. So, by Corollary \ref{PIprima3}, it does 
not satisfy the PI property.
\end{proof}

\subsection*{The algebra $U$}

The algebra $U$, defined in Example 2.7.7 of \cite{Berger1992}, is the $\C$-algebra 
generated by $x_i, y_i, z_i$, for $1 \le i \le n$, 
subject to the relations

\begin{align*}
x_j x_i &= x_i x_j, &y_j y_i = y_i y_j,  \qquad &z_j z_i= z_i z_j, 
& 1 \le i,j \le n,\\[2pt]
y_j x_i &= q^{\delta_{ij}}\, x_i y_j,  
&z_j x_i = q^{-\delta_{ij}}\, x_i z_j, &   & 1 \le i,j \le n,\\[2pt]
z_j y_i &= y_i z_j, &  &   &i \ne j,\\[2pt]
z_i y_i &= q^{2}\, y_i z_i - q^{2}\, x_i^{2}, &  &  & 1 \le i \le n,
\end{align*}

where $q \in \mathbb{C}^*$. Note that $U$ is a bijective skew PBW extension of 
$\mathbb{C}[x_1,\dots,x_n]$:
\[
U \cong \sigma\!\bigl(\mathbb{C}[x_1,\dots,x_n]\bigr)
\langle y_1,\dots,y_n;\; z_1,\dots,z_n\rangle.
\]
In \cite{LezamaVenegas2020} the following important result about $Z(U)$ was proved.
\begin{theorem}[{\cite[ Theorem 3.4]{LezamaVenegas2020}}]
If $q$ is not a root of unity, then $Z(U) = \C$. If $q$ is a root of unity of 
order $l \ge 2$, then $x_1^l,\dots,x_n^l$ are central elements of $U$ and 
$\C[x_1^l,\dots,x_n^l] \subseteq Z(U)$.
\end{theorem}

\begin{proposition}
If $q$ is not a root of unity, then the algebra $U$ does not satisfy the PI 
property.
\end{proposition}

\begin{proof}
Since $U \cong \sigma\!\bigl(\mathbb{C}[x_1,\dots,x_n]\bigr)\langle 
y_1,\dots,y_n;\; z_1,\dots,z_n\rangle$ is a bijective skew PBW extension over 
the domain $\mathbb{C}[x_1,\dots,x_n]$, Lemma~\ref{PrimoPBW} implies that $U$ 
is a prime ring. By the theorem above, $Z(U) = \C$ and, by 
Corollary \ref{PIprima3}, this algebra does not satisfy the PI property.
\end{proof}

\subsection*{Quadratic algebras in three variables}

A quadratic algebra in three variables is a $\K$-algebra $A$ generated by 
$x, y, z$ subject to the relations
\[
\begin{aligned}
yx &= xy + a_1 zx + a_2 y^2 + a_3 yz + \xi_1 z^2,\\
zx &= xz + \xi_2 y^2 + a_5 yz + a_6 z^2,\\
zy &= yz + a_4 z^2.
\end{aligned}
\]
When $a_3 = a_5 = 0$ (which implies $a_2 = a_6 = 0$), the defining relations 
reduce to
\[
yx = xy + a_1 zx + \xi_1 z^2, \qquad zx = xz, \qquad zy = yz + a_4 z^2,
\]
yielding a subfamily of algebras that are skew PBW extensions of the form 
$A_2 \cong \sigma\!\bigl(\K[x,z]\bigr)\langle y\rangle$.
With regard to the center of these algebras, the following fact was proven in \cite{LezamaVenegas2020}.
\begin{theorem}[{\cite[Theorem 3.9]{LezamaVenegas2020}}]
Let $\K$ be an ordered field. The quadratic algebra $A_2$ generated by $x, y, z$ 
with relations $yx = xy + axz + bz^2$, $zx = xz$, $zy = yz + cz^2$, where 
$a, b, c \in \K^*$ and $ac < 0_{\K}$, has trivial center.
\end{theorem}

\begin{proposition}
Let $\K$ be an ordered field with $a, b, c \in \K^*$ and $ac < 0_{\K}$. Then $A_2$ 
does not satisfy the PI property.
\end{proposition}

\begin{proof}
Since $A_2 \cong \sigma\!\bigl(\K[x,z]\bigr)\langle y\rangle$ is a bijective 
skew PBW extension over the domain $\K[x,z]$, Lemma \ref{PrimoPBW} implies that 
$A_2$ is a prime ring. By the theorem above, $Z(A_2) = \K$ and, by 
Corollary \ref{PIprima3}, the algebra $A_2$ does not satisfy the PI property.
\end{proof}

\subsection*{Witten's deformation of $U(\mathfrak{sl}(2,\K))$}

Witten introduced a seven-parameter deformation of the universal enveloping 
algebra $U(\mathfrak{sl}(2,\K))$, depending on a 7-tuple 
$\boldsymbol{\xi} = (\xi_1,\ldots,\xi_7)$, subject to the relations
\[
xz - \xi_1 zx = \xi_2 x, \qquad
zy - \xi_3 yz = \xi_4 y, \qquad
yx - \xi_5 xy = \xi_6 z^2 + \xi_7 z.
\]
The resulting algebra is denoted $W(\boldsymbol{\xi})$ and   turns out to be isomorphic to a bijective skew PBW extension of the form $\sigma\bigl(\sigma(\K[x])\langle z\rangle\bigr)\langle 
y\rangle$.

A particular case was studied by Levandovskyy in \cite{Levandovskyy2005}, under 
the assumption $\xi_1\xi_3\xi_5 \ne 0$. In this case the algebra is generated 
by $x, y, z$ subject to
\[
yx = cxy + bz^2 + z, \qquad
zx = \frac{1}{a}xz - \frac{1}{a}x, \qquad
zy = ayz + y,
\]
where $a, b, c \in \K^{*}$.

\begin{proposition}
The particular case of Witten's deformation described above does not satisfy the 
PI property.
\end{proposition}

\begin{proof}
By Theorem 3.10 of \cite{LezamaVenegas2020}, the center of this algebra is 
trivial. Applying Lemma \ref{PrimoPBW} twice yields that it is a prime ring and, 
by Corollary \ref{PIprima3}, it does not satisfy the PI property.
\end{proof}

The following example illustrates that the characteristic of the base field plays 
a relevant role in the study of the PI property of an algebra.

\subsection*{Jordan plane}
Let $\K$ be a field. The Jordan plane $\Lambda_2(\K)$ is the $\K$-algebra generated 
by $x$ and $y$ subject to the relation $yx = xy + x^2$, that is,
\[
\Lambda_2(\K) = \K\langle x,y\rangle/\langle yx - xy - x^2\rangle.
\]
This algebra can be seen as a skew PBW extension of $\K[x]$: 
$\Lambda_2(\K) \cong \sigma(\K[x])\langle y\rangle$, with $yx = xy + x^2$.

\begin{theorem}[{\cite[Theorem 2.2-(i)]{Shirikov2005}}]\label{centro-plano-jordan} Let $\Lambda_2(\K)$ be the algebra defined above. Then, the following holds:
\begin{enumerate}[\rm (i)]
  \item If $\mathrm{char}(\K) = 0$, then $Z(\Lambda_2(\K)) = \K$.
  \item If $\mathrm{char}(\K) = p > 0$ is prime, then $Z(\Lambda_2(\K))$ 
  is the subalgebra of $\Lambda_2(\K)$ generated by $x^p$ and $y^p$.
\end{enumerate}
\end{theorem}
\noindent
The previous result allows us to demonstrate the following facts about the PI property of $\Lambda_2(\K)$.

\begin{theorem}\label{Propiedad-pi-plano-de-jordan} Let $\Lambda_2(\K)$ as before. 
\begin{enumerate}[\rm (i)] 
  \item If $\mathrm{char}(\K) = 0$, then $\Lambda_2(\K)$ does not satisfy 
  the PI property.
  \item If $\mathrm{char}(\K) = p > 0$ is prime, then $\Lambda_2(\K)$ 
  satisfies the PI property.
\end{enumerate}
\end{theorem}

\begin{proof}
(i) Suppose $\mathrm{char}(\K) = 0$ and that $\Lambda_2(\K)$ satisfies the 
PI property. Since $\K[x]$ is an integral domain and hence a prime ring,  Lemma \ref{PrimoPBW} implies that $\Lambda_2(\K)$ is a prime skew PBW extension. 
By Theorem \ref{PIprima2}, one has $Z(\Lambda_2(\K)) \ne \K$, contradicting 
Theorem \ref{centro-plano-jordan}-(i).

(ii) If $\mathrm{char}(\K) = p > 0$, Theorem \ref{centro-plano-jordan} 
gives $Z(\Lambda_2(\K)) = \K[x^p, y^p]$. Therefore $\Lambda_2(\K)$ is finitely 
generated as a $Z(\Lambda_2(\K))$-module, for \cite[\S 13.1.13, Corollary]{mcconnell} this algebra satisfies the PI property.
\end{proof}

The following table summarizes the algebras studied in this subsection that do 
not satisfy the PI property when none of their parameters $q$ is a root of 
unity, working over a field $\K$ of characteristic zero.

\begin{table}[h]
\centering
\renewcommand{\arraystretch}{1.2}
\begin{tabular}{|l|c|}
\hline
\textbf{Algebra $A$} & $\boldsymbol{Z(A)}$ \\
\hline
\hline
Extended Weyl algebra $B_n(\K)$ & $\K$ \\
Shift operator algebra $S_h$ & $\K$ \\
Mixed algebra $D_h$ & $\K$ \\
Linear partial shift operators 
  $\K[t_1,\ldots,t_n][E_1,\ldots,E_n]$ & $\K$ \\
Linear partial differential operators 
  $\K[t_1,\ldots,t_n][\partial_1,\ldots,\partial_n]$ & $\K$ \\
Linear partial difference operators 
  $\K[t_1,\ldots,t_n][\Delta_1,\ldots,\Delta_n]$ & $\K$ \\
Linear partial $q$-dilation operators 
  $\K[t_1,\ldots,t_n][H^{(q)}_1,\ldots,H^{(q)}_n]$ & $\K$ \\
Linear partial $q$-differential operators 
  $\K[t_1,\ldots,t_n][D^{(q)}_1,\ldots,D^{(q)}_n]$ & $\K$ \\
Additive analogue of the Weyl algebra $A_n(q_1,\ldots,q_n)$ & $\K$ \\
Woronowicz algebra $W_\nu(\mathfrak{sl}(2,\K))$ & $\K$ \\
Algebra $U$ & $\K$ \\
Particular Witten deformation of $U(\mathfrak{sl}(2,\K))$ & $\K$ \\
Quantum symplectic space 
  $\mathcal{O}_q(\mathfrak{sp}(\K^{2n}))$ & $\K$ \\
Jordan plane $\Lambda_2(\K)$ ($\mathrm{char}(\K) = 0$) & $\K$ \\
Quantum plane $\mathcal{O}_q(\K^2)$ ($q$ not a root of unity) & $\K$ \\
Quadratic algebras in three variables $A_2$ & $\K$ \\
\hline
\end{tabular}
\caption{Skew PBW extensions with trivial center that do not satisfy the PI 
property.}
\end{table}

\section{The PI property of algebras over fields of positive characteristic}
\label{charpos}

In the preceding sections, the base field $\K$ has been assumed, unless otherwise 
stated, to have characteristic zero. In this section we study the PI property of 
various algebras over fields of positive characteristic. The main theoretical 
framework is provided by the work of Brown and Zhang \cite{BrownZhang2022}, who 
studied iterated Hopf--Ore extensions over an algebraically closed field of 
positive characteristic $p$ and obtained, among other results, necessary and 
sufficient conditions for such algebras to satisfy the PI property.

\begin{lemma}[{\cite[Lemma 3.2]{BrownZhang2022}}]\label{LemaBrownZhang3.2}
Let $B = A[x;\sigma,\delta]$ be an Ore extension of an affine $\K$-algebra $A$.
\begin{enumerate}[\rm (i)]
  \item Suppose that the following conditions hold:
        \begin{enumerate}[\rm (a)]
          \item The $\K$-algebra $A$ is a prime ring that is a finitely generated module over 
          $Z(A)$;
          \item $Z(A)$ is an affine $\K$-algebra;
          \item the restriction $\sigma|_{Z(A)}$ has finite order;
          \item  the characteristic of $\K$ is positive.
        \end{enumerate}
        Then $B$ is a prime Noetherian $\K$-algebra satisfying the PI property, 
        and $A$ is a finite module over a subalgebra of its center on which 
        $\sigma$ acts as the identity and $\delta$ acts as the zero derivation.
\item If, in addition to hypotheses \emph{(ia)}--\emph{(id)}, one 
  assumes:
        \begin{enumerate}
          \item[\rm (e)] $A$ is a maximal order;
        \end{enumerate}
        then $B$ is a maximal order, is a finitely generated module over $Z(B)$, 
        and is a normal affine $\K$-algebra.
\end{enumerate}
\end{lemma}

\begin{example}[First Weyl algebra]\label{Weylcp}
Let $\K$ be a field of positive characteristic $p > 0$ and consider the first 
Weyl algebra $A_1(\K) = \K[y][x;\mathrm{id},\frac{d}{dy}]$. In contrast to the 
characteristic-zero case, where $Z(A_1(\K)) = \K$ and the algebra does not satisfy 
the PI property, when $\mathrm{char}(\K) = p > 0$, the algebra $A_1(\K)$ satisfies 
the hypotheses of Lemma \ref{LemaBrownZhang3.2} and turns out be a prime 
Noetherian $\K$-algebra satisfying the PI property.
\end{example}

\begin{example}[Quantum plane]\label{ej:plano-cuantico}
Let $\K$ be a field of positive characteristic $p > 0$ and let $q \in \K^{*}$. 
Consider the quantum plane $\mathcal{O}_q(\K^2) = \K\langle x,y \mid yx = qxy
\rangle \cong \K[y][x;\sigma]$, with $\sigma(y) = qy$. The algebra 
$\mathcal{O}_q(\K^2)$ satisfies the hypotheses of Lemma \ref{LemaBrownZhang3.2} 
if and only if $q$ is a root of unity, since in that case $\sigma|_{Z(\K[y])}$ 
has finite order. Under this condition, $\mathcal{O}_q(\K^2)$ is a prime 
Noetherian $\K$-algebra satisfying the PI property. This criterion coincides with 
the one established in characteristic zero: the quantum plane satisfies the PI 
property if, and only if, $q$ is a root of unity, regardless of the characteristic 
of the base field.
\end{example}
The following property, proven in \cite{BrownZhang2022}, will allow us to establish necessary and sufficient 
conditions for the PI property in several classes of algebras that admit a 
presentation as iterated Ore extensions, including quasi-commutative skew PBW 
extensions and double Ore extensions. Hereinafter, the base field 
$\K$ is assumed to have positive characteristic. A domain $R$ is said to be \emph{homologically homogeneous} if all simple right $R$-modules have the same projective dimension, moreover, if $R$ is a noetherian $\K$-algebra such that $\operatorname{GKdim}(R)<\infty$, then $R$ is said to be \emph{GK-Cohen--Macaulay} if, for every nonzero finitely generated right $R$-module $M$, one has
\[
j_R(M)+\operatorname{GKdim}(M)=\operatorname{GKdim}(R),
\]
where
\[
j_R(M):=\inf\{\,i\geq 0\mid \operatorname{Ext}^i_R(M,R)\neq 0\,\}.
\]

\begin{theorem}[{\cite[Theorem 3.3]{BrownZhang2022}} ]\label{Iteradas-3.3}
Suppose that $\K$ has positive characteristic $p$, let $n$ be a non-negative 
integer, and let $R$ be an iterated Ore extension of $n$ steps
\begin{equation}\label{eq:iterated-Ore}
R \;\cong\; \K[x_1][x_2;\sigma_2,\delta_2]\cdots[x_n;\sigma_n,\delta_n].
\end{equation}
For $i = 1,\dots,n$, denote $R(i) := \K\langle x_1,\dots,x_i\rangle$. Then,
\begin{enumerate}[\rm (i)]
  \item The ring $R$ satisfies the PI property if and only if the restriction 
  $\sigma_i\big|_{Z(R(i-1))}$ has finite order for all $i = 2,\dots,n$.

  \item If $R$ satisfies the PI property, then $R$ is a finitely generated 
  module over its center, which is a normal affine domain over $K$. Moreover, $R$ 
  is a homologically homogeneous and GK-Cohen--Macaulay domain, with
  \[
    \operatorname{gldim}(R) = \operatorname{GKdim}(R) = n.
  \]
\end{enumerate}
\end{theorem}
For quasi-commutative skew PBW extensions the following property holds.
\begin{theorem}[{\cite[Theorem 3.1.4]{skewbook}}]\label{thm:quasi-commutative-iterated}
Let $A$ be a quasi-commutative skew PBW extension of a ring $R$. Then $A$ is 
isomorphic to an iterated skew polynomial ring of endomorphism type of the form
\[
B := R[z_1;\theta_1]\cdots[z_n;\theta_n],
\]
where
\[
\left\{
\begin{aligned}
&\theta_1 := \sigma_1;\\
&\theta_j \colon R[z_1;\theta_1]\cdots[z_{j-1};\theta_{j-1}]
   \longrightarrow R[z_1;\theta_1]\cdots[z_{j-1};\theta_{j-1}],\\
&\theta_j(z_i) := c_{ij}z_i,\ 1 \le i < j \le n,\quad
 \theta_j(r) := \sigma_j(r),\ \text{for all } r \in R.
\end{aligned}
\right.
\]
\end{theorem}

\begin{theorem}\label{PI-qc-skewPBW}
Let $\K$ be a field of positive characteristic $p$ and let 
$A = \sigma(\K[x_1])\langle x_2,\dots,x_n\rangle$ be a quasi-commutative skew 
PBW extension of $\K[x_1]$. For $i = 1,\dots,n$, denote
\[
A(i) := \K\langle x_1,\dots,x_i\rangle.
\]
Then,
\begin{enumerate}[\rm (i)]
  \item The extension $A$ satisfies the PI property if and only if, for all 
  $i = 2,\dots,n$, the restriction $\theta_i|_{Z(A(i-1))}$ has finite order, 
  where the $\theta_i$ are as in Theorem \ref{thm:quasi-commutative-iterated}.

  \item If $A$ satisfies the PI property, then $A$ is a finitely generated 
  module over its center $Z(A)$, which is a normal affine domain over $\K$. 
  Moreover, $A$ is a homologically homogeneous and GK-Cohen--Macaulay domain, 
  with
  \[
    \operatorname{gldim}(A) = \operatorname{GKdim}(A) = n.
  \]
\end{enumerate}
\end{theorem}

\begin{proof}
Since $A$ is a quasi-commutative skew PBW extension of $\K[x_1]$, 
Theorem~\ref{thm:quasi-commutative-iterated} yields an isomorphism 
$A \cong \K[x_1][z_2;\theta_2]\cdots[z_n;\theta_n]$ as iterated skew polynomial 
rings of endomorphism type. Since $\K$ has positive characteristic, 
Theorem \ref{Iteradas-3.3} applies and yields the stated results.
\end{proof}

\begin{example}
Let $\K$ be a field of positive characteristic $p$. The multiplicative analogue 
of the Weyl algebra, realized as a skew PBW extension over $\K[x_1]$, see 
\cite[Example 3.5-(b)]{LezamaReyes2014}, 
\[
\mathcal{O}_n(\lambda_{ji}) \cong \sigma(\K[x_1])\langle x_2,\ldots,x_n\rangle,
\]
satisfies the hypotheses of Theorem \ref{PI-qc-skewPBW} with
\[
\mathcal{O}_n(\lambda_{ji}) \cong \K[x_1]\,[x_2;\sigma_2]\,[x_3;\sigma_3]
\cdots[x_n;\sigma_n],
\]
where $\sigma_j(x_i) = \lambda_{ji}\, x_i$ for $i < j$. Therefore, 
$\mathcal{O}_n(\lambda_{ji})$ satisfies the PI property if and only if all 
parameters $\lambda_{ji}$ are roots of unity.
\end{example}

\subsection*{Two-parameter quantum Heisenberg algebra $\mathcal{H}_{p,q}$}
\noindent
The two-parameter quantum Heisenberg algebra $\mathcal{H}_{p,q}$ is the 
$\K$-algebra generated by $x, y, t$ subject to the relations
\[
tx = p^{-1}xt, \qquad ty = pyt, \qquad yx - qxy = t,
\]
where $p, q \in \K^*$. This algebra admits a presentation as an iterated Ore 
extension  of the form
\[
\mathcal{H}_{p,q} \cong \K[t][x;\sigma_x][y;\sigma_y,\delta_y],
\]
where $\sigma_x(t) = pt$, $\sigma_y(t) = p^{-1}t$, $\sigma_y(x) = qx$, 
$\delta_y(t) = 0$, and $\delta_y(x) = t$, see \cite[Proposition 3.4]{Gaddis2016}.
Under these conditions, the following statement holds (compare with \cite[Proposition 3.4]{gomezgallego1}).
\begin{proposition}
Let $\K$ be a field of positive characteristic $l$ and let $p, q \in \K^*$ with 
$pq \neq 1$. The algebra $\mathcal{H}_{p,q}$ satisfies the PI property if and 
only if the parameters $p$ and $q$ are roots of unity.
\end{proposition}

\begin{proof}
Define
\[
R(1) := \K[t], \qquad R(2) := R(1)[x;\sigma_x], \qquad 
R(3) := R(2)[y;\sigma_y,\delta_y] = \mathcal{H}_{p,q}.
\]
Suppose that $\mathcal{H}_{p,q}$ satisfies the PI property. By 
Theorem \ref{Iteradas-3.3}, the restriction $\sigma_x|_{Z(R(1))}$ has finite order. Since 
$Z(\K[t]) = \K[t]$ and $\sigma_x(t) = pt$, there exists $m \ge 1$ such that 
$\sigma_x^m = \mathrm{id}|_{\K[t]}$; in particular, $p^m t = t$, so $p^m = 1$ 
and $p$ is a root of unity. Similarly, $\sigma_y|_{Z(R(2))}$ has finite order. 
Let $m = \mathsf{ord}(p)$; in $R(2)$ the relation $xt = ptx$ implies that 
$x^m t = p^m t x^m = t x^m$, so $x^m \in Z(R(2))$. Since 
$\sigma_y|_{Z(R(2))}$ has finite order, there exists $r \ge 1$ such that 
$\sigma_y^r|_{Z(R(2))} = \mathrm{id}_{Z(R(2))}$; in particular,
\[
x^m = \sigma_y^r(x^m) = (qx)^{mr} = q^{mr}x^m,
\]
which implies $q^{mr} = 1$ and hence $q$ is a root of unity.

Conversely, suppose that $p$ and $q$ are roots of unity. Let 
$m = \mathsf{ord}(p)$ and $n = \mathsf{ord}(q)$. Since $\sigma_x(t) = pt$ and 
$\sigma_x^m(t) = p^m t = t$, we have $\sigma_x^m = \mathsf{id}$ on $\K[t]$, so 
$\sigma_x|_{Z(R(1))}$ has finite order. In $R(2) = \K[t][x;\sigma_x]$, the 
relation $xt = ptx$ gives the identities
\[
x\,t^b=p^b t^b x,
\qquad
x^a t=p^a t x^a,
\qquad
t\,x^a=p^{-a}x^a t,
\qquad
x^a t^b=p^{ab}t^b x^a,
\]
for all $a$, $b\geq 0$.
Thus, a monomial $t^b x^a$ commutes with $t$ if and only if $p^a = 1$, and commutes with $x$ 
if and only if $p^b = 1$. Since $m = \mathsf{ord}(p)$, this latter holds if and only 
if $m \mid a$ and $m \mid b$. Therefore $Z(R(2)) = \K[t^m, x^m]$. Let 
$L := \mathsf{lcm}(m,n)$. Since $\sigma_y(t^m) = p^{-m}t^m$ and 
$\sigma_y(x^m) = q^m x^m$, we have
\[
\sigma_y^L(t^m) = p^{-mL}t^m = t^m, \qquad \sigma_y^L(x^m) = q^{mL}x^m = x^m.
\]
Hence $\sigma_y^L|_{Z(R(2))} = \mathrm{id}_{Z(R(2))}$, so 
$\sigma_y|_{Z(R(2))}$ has finite order. By Theorem \ref{Iteradas-3.3}, the algebra
$R(3) = \mathcal{H}_{p,q}$ satisfies the PI property.
\end{proof}

\subsection*{Double Ore extensions}

In \cite{CarvalhoLopesMatczuk2011}, Carvalho, Lopes, and Matczuk  established 
necessary and sufficient conditions for a double Ore extension to be an iterated 
Ore extension, which allows us to apply Theorem~\ref{Iteradas-3.3} to this 
class of algebras.

\begin{proposition}[{\cite[Proposition 1.2]{CarvalhoLopesMatczuk2011}}] \label{Prop12}
Consider the right double Ore extension $B = \K_P[y_1,y_2;\sigma',\delta',\tau]$. 
Then $B \cong \K[x_1][x_2;\sigma_2,d_2]$ is an iterated Ore extension, where
\[
\sigma_2(x_1) = p_{12}x_1 + \tau_2
\]
and $d_2$ is the $\sigma_2$-derivation of $\K[x_1]$ given by
\[
d_2(x_1) = p_{11}x_1^2 + \tau_1 x_1 + \tau_0.
\]
Moreover, $B$ is a double Ore extension of $\K$ if and only if $p_{12} \neq 0$.
\end{proposition}
An immediate consequence of the above assumptions is the following.
\begin{theorem}\label{PI-double-Ore}
Let $\K$ be a field of positive characteristic $p$ and let
\[
B = \K_P[y_1,y_2;\sigma',\delta',\tau]
\]
be a right double Ore extension of $\K$ with 
$\tau = (\tau_0,\tau_1,\tau_2) \in \K^3$. Then:
\begin{enumerate}[\rm (i)]
  \item The algebra $B$ satisfies the PI property if and only if the restriction 
  $\sigma_2|_{Z(\K[x_1])}$ (where $\sigma_2$ is as in Proposition \ref{Prop12} 
  and $Z(\K[x_1]) = \K[x_1]$) has finite order.

  \item If $B$ satisfies the PI property, then $B$ is a finitely generated 
  module over its center $Z(B)$, which is a normal affine domain over $\K$. 
  Moreover, $B$ is a homologically homogeneous and GK-Cohen--Macaulay domain, 
  with
  \[
    \operatorname{gldim}(B) = \operatorname{GKdim}(B) = 2.
  \]
\end{enumerate}
\end{theorem}

\begin{proof}
By Proposition \ref{Prop12}, there is a $\K$-algebra isomorphism 
$B \cong \K[x_1][x_2;\sigma_2,d_2]$. Applying Theorem \ref{Iteradas-3.3} with 
$R(1) := \K[x_1]$ yields the stated results.
\end{proof}

\begin{theorem}[{\cite[Theorem 2.2]{CarvalhoLopesMatczuk2011}}]\label{teo:Carvalho-2-2}
Let $A, B$ be $\K$-algebras such that $B$ is an extension of $A$. Suppose 
$P = \{p_{12},p_{11}\} \subseteq \K$, $\tau = \{\tau_0,\tau_1,\tau_2\} \subseteq A$, 
$\sigma \colon A \to M_{2\times 2}(A)$ is an algebra morphism, and 
$\delta \colon A \to M_{2\times 1}(A)$ is a $\sigma$-derivation.
\begin{enumerate}[\rm (a)]
\item The following conditions are equivalent:
  \begin{enumerate}[\rm (i)]
  \item $B = A_P[y_1,y_2;\sigma,\delta,\tau]$ is a right double Ore extension 
  of $A$ that can be presented as an iterated Ore extension 
  $A[y_1;\sigma_1,d_1][y_2;\sigma_2,d_2]$.

  \item $B = A_P[y_1,y_2;\sigma,\delta,\tau]$ is a right double Ore extension 
  of $A$ with $\sigma_{12} = 0$.

  \item $B = A[y_1;\sigma_1,d_1][y_2;\sigma_2,d_2]$ is an iterated Ore 
  extension such that
  \[
    \sigma_2(A) \subseteq A, \qquad \sigma_2(y_1) = p_{12}y_1 + \tau_2,
  \]
  \[
    d_2(A) \subseteq Ay_1 + A, \qquad d_2(y_1) = p_{11}y_1^2 + \tau_1 y_1 + \tau_0,
  \]
  for certain $p_{ij} \in \K$ and $\tau_i \in A$. The maps 
  $\sigma, \delta, \sigma_i, d_i$, $i = 1,2$, are related by
  \[
    \sigma = \begin{bmatrix} \sigma_1 & 0\\[2pt] \sigma_{21} & \sigma_2|_A 
    \end{bmatrix},
    \qquad
    \delta(a) = \begin{bmatrix} d_1(a)\\[2pt] d_2(a) - \sigma_{21}(a)y_1 
    \end{bmatrix},
    \qquad \text{for all } a \in A.
  \]
  \end{enumerate}

\item If any of the equivalent conditions from \emph{(a)} holds, then $B$ is 
a double Ore extension of $A$ if and only if $\sigma_{11}:=\sigma_1$ and $\sigma_{22}:=\sigma_{2}|_A$ are 
automorphisms of $A$ and $p_{12} \neq 0$.
\end{enumerate}
\end{theorem}

\begin{theorem}\label{PI-double-Ore2}
Let $\K$ be a field of positive characteristic $p$, let $A = \K[t]$, and let 
$B = A_P[y_1,y_2;\sigma,\delta,\tau]$ be a right double Ore extension of $A$ 
with $\sigma_{12} = 0$ ($\sigma, \delta, \tau$, and $P$ as in 
Theorem~\ref{teo:Carvalho-2-2}). Then,
\begin{enumerate}[\rm (i)]
  \item The algebra $B$ satisfies the PI property if and only if the restrictions 
  $\sigma_1|_{\K[t]}$ and $\sigma_2|_{Z(\K\langle t,y_1\rangle)}$ have finite order.

  \item If $B$ satisfies the PI property, then $B$ is a finitely generated 
  module over its center $Z(B)$, which is a normal affine domain over $\K$. 
  Moreover, $B$ is a homologically homogeneous and GK-Cohen--Macaulay domain, 
  with
  \[
    \operatorname{gldim}(B) = \operatorname{GKdim}(B) = 3.
  \]
\end{enumerate}
\end{theorem}

\begin{proof}
By Theorem \ref{teo:Carvalho-2-2}, the algebra $B$ admits the iterated Ore extension 
presentation $$\K[t][y_1;\sigma_1,d_1][y_2;\sigma_2,d_2].$$ Applying 
Theorem \ref{Iteradas-3.3} with $R(1):= \K[t]$ and 
$R(2):= \K\langle t,y_1\rangle$ yields the stated results.
\end{proof}

\begin{example}[{\cite[p. 2685]{Zhang}}] \label{oreejm4}
The $\K$-algebra $A^4 := A^4(a,b,c)$, where $a,b,c \in \K$ and $b \neq -1$, is 
defined by the relations
\begin{align*}
x_2 x_1 & = x_1 x_2 + x_1^2 + atx_1 + \dfrac{b}{1+b}tx_2 + ct^2,\\[6pt]
x_1 t & = tx_1 + bt^2,\\[6pt]
x_2 t & = (2+2b^{-1})tx_1 + cx_2.
\end{align*}
This is a right double Ore extension of $\K[t]$, with
\[
\sigma(t) = \begin{pmatrix} t & 0 \\ (2+2b^{-1})t & t \end{pmatrix}, \qquad
\delta(t) = \begin{pmatrix} bt^2 \\ 0 \end{pmatrix},
\]
parameter $P = (1,1)$, and tail $\tau = \left\{at, \frac{b}{1+b}t, ct^2\right \}$. Since 
$\sigma_{12} = 0$, Theorem \ref{PI-double-Ore2} applies and, assuming the base 
field has positive characteristic, it follows that $B$ satisfies the PI property if and only if the restriction 
$\sigma_2|_{Z(\K\langle t,x_1\rangle)}$ has finite order.
\end{example}
The following result from \cite{CarvalhoLopesMatczuk2011} provides an additional characterization so that a double Ore extension can be written as an iterated Ore extension. This will allow us to determine the conditions under which a particular class of such algebras satisfies the PI property.
\begin{theorem}[{\cite[Theorem 2.4]{CarvalhoLopesMatczuk2011}}]\label{thm:2.4}
Let $B = A_P[y_1,y_2;\sigma,\delta,\tau]$ be a right double Ore extension of the 
$\K$-algebra $A$, where $P = \{p_{12},p_{11}\} \subseteq \K$, 
$\tau = \{\tau_0,\tau_1,\tau_2\} \subseteq A$, 
$\sigma \colon A \to M_{2\times 2}(A)$ is an algebra morphism, and 
$\delta \colon A \to M_{2\times 1}(A)$ is a $\sigma$-derivation. Then $B$ can 
be presented as the iterated Ore extension 
$A[y_2;\sigma'_2,d_2'][y_1;\sigma'_1,d_1']$ if and only if $\sigma_{21} = 0$, 
$p_{12} \neq 0$, and $p_{11} = 0$.
\end{theorem}

\begin{theorem}\label{PI-double-Ore3}
Let $\K$ be a field of positive characteristic $p$, let $A = \K[t]$, and let 
$B = A_P[y_1,y_2;\sigma,\delta,\tau]$ be a right double Ore extension of $A$ 
with $\sigma_{21} = 0$, $p_{12} \neq 0$, and $p_{11} = 0$ ($\sigma, \delta, 
\tau$, and $P$ as in Theorem \ref{thm:2.4}). Then,
\begin{enumerate}[\rm (i)]
  \item The algebra $B$ satisfies the PI property if and only if the restrictions 
  $\sigma'_2|_{\K[t]}$ and $\sigma'_1|_{Z(\K\langle t,y_2\rangle)}$ have finite 
  order.

  \item If $B$ satisfies the PI property, then $B$ is a finitely generated 
  module over its center $Z(B)$, which is a normal affine domain over $\K$. 
  Moreover, $B$ is a homologically homogeneous and GK-Cohen--Macaulay domain, 
  with
  \[
    \operatorname{gldim}(B) = \operatorname{GKdim}(B) = 3.
  \]
\end{enumerate}
\end{theorem}

\begin{proof}
By Theorem \ref{thm:2.4}, $B$ admits the iterated Ore extension presentation 
$\K[t][y_2;\sigma'_2,d_2'][y_1;\sigma'_1,d_1']$. Applying 
Theorem \ref{Iteradas-3.3} with $R(1) := \K[t]$ and 
$R(2) := \K\langle t,y_2\rangle$ yields the stated results.
\end{proof}

\begin{example}[{\cite[p. 2685]{Zhang}}]
The $\K$-algebra $A^3 := A^3(a)$, where $a \in \K$ and $a \neq 0$, is defined by 
the relations
\begin{align*}
x_2 x_1 & = x_1 x_2,\\
x_1 t & = atx_1 + tx_2,\\
x_2 t & = atx_2.
\end{align*}
This is a right double Ore extension of $\K[t]$, with
\[
\sigma(t) = \begin{pmatrix} at & 0 \\ 0 & at \end{pmatrix}, \qquad
\delta(t) = \begin{pmatrix} 0 \\ 0 \end{pmatrix},
\]
parameter $P = (1,0)$, and tail $\tau = \{0,0,0\}$. Although $\sigma_{21} = 0$, 
$p_{12} \neq 0$, and $p_{11} = 0$, so that Theorem \ref{thm:2.4} applies in 
principle, the maps $\sigma'_2$ and $\sigma'_1$ arising from that theorem are 
not determined explicitly from the given presentation, and consequently the 
criterion of Theorem \ref{PI-double-Ore3} cannot be applied directly in this 
case.
\end{example}

Zhang and Zhang \cite{Zhang} introduce the following right double Ore extension 
over $A = \K[x]$.

\begin{example}[{\cite[p. 2685]{Zhang}}]
Let $A = \K[x]$. The algebra $B^2 := B^2(a,b,c)$, with $a,b,c \in \K$ and 
$b \neq 0$, is defined by the relations
\begin{align*}
y_2 y_1 &= -y_1 y_2 + ax^2,\\
y_1 x &= b^{-1}xy_2 + cx^2,\\
y_2 x &= bxy_1 + (-bc)x^2.
\end{align*}
The morphism $\sigma$ is determined by
\[
\sigma(x) = \begin{pmatrix} 0 & b^{-1}x \\ bx & 0 \end{pmatrix},
\]
the derivation by
\[
\delta(x) = \begin{pmatrix} cx^2 \\ -bc\,x^2 \end{pmatrix},
\]
the parameter $P = (-1,0)$, and the tail $\tau = \{0,0,ax^2\}$. Zhang and Zhang 
claim that, for most values of $(a,b,c)$ (for instance, for $(a,b,c) = (1,1,0)$), 
the algebra $B^2(a,b,c)$ is not an iterated Ore extension of $\K[x]$. However, 
this assertion does not hold in general: as shown below, when the base field has 
characteristic $2$, this $B^2$ does admit a presentation as an iterated Ore extension.
\end{example}

\begin{proposition}[{\cite[Theorem 2.7]{CarvalhoLopesMatczuk2011}}]\label{prop:2.7}
Let $a,b,c \in \K$ with $b \neq 0$. The algebra $B^2 = B^2(a,b,c)$ has the 
following properties:
\begin{enumerate}[\rm (a)]
\item If $\mathrm{char}(\K) = 2$, then $B^2$ is the differential operator 
algebra
\[
B^2 \cong \K[x,z][y_2;d],
\]
where $d$ is the derivation of $\K[x,z]$ determined by
\[
d(x) = xz - bcx^2, \qquad d(z) = abx^2.
\]
In particular, such $B^2$ can be presented as an iterated Ore extension over 
$A = \K[x]$.
\item The algebra $B^2$ is a Noetherian domain.
\end{enumerate}
\end{proposition}

\begin{theorem}
Let $\K$ be a field of characteristic $2$. Then the algebra $B^2$ satisfies the 
PI property. Moreover, $B^2$ is a finitely generated module over its center, 
which is a normal affine domain over $\K$; furthermore, $B^2$ is a homologically 
homogeneous and GK-Cohen--Macaulay domain, with
\[
\operatorname{gldim}(B^2) = \operatorname{GKdim}(B^2) = 3.
\]
\end{theorem}

\begin{proof}
By Proposition \ref{prop:2.7}, $B^2$ admits the iterated Ore extension 
presentation
\[
B^2 \cong \K[x,z][y_2;d] = \K[x][z;\mathrm{id},0][y_2;\mathrm{id},d].
\]
Since the base field has characteristic $2$ and all endomorphisms involved are 
the identity, part (i) of Theorem \ref{Iteradas-3.3} implies that $B^2$ 
satisfies the PI property. Part (ii) of the same theorem then yields the stated 
homological properties, with $\operatorname{gldim}(B^2) = \operatorname{GKdim}(B^2) = 3$.
\end{proof}

\begin{remark}
To the best of the authors' knowledge, and based on the bibliographic review 
carried out, no criterion is currently available in the literature that explicitly 
describes the conditions under which the algebra $B^2$ satisfies the PI property 
when the base field has characteristic zero.
\end{remark}
\subsection*{The algebra $B_q(f)$}

With the aim of exhibiting a family of algebras with trivial Ozone group, Gaddis 
and Yee recently introduced the algebra $B_q(f)$ in \cite{GaddisYee2025}. For 
$q \in \K^{*}$ and $f \in \K[t]$, the algebra $B_q(f)$ is the $\K$-algebra freely 
generated by $u, v, w$, subject to the relations
\[
uv = qvu, \qquad wu = quw + f(v), \qquad wv = q^{-1}vw + f(u).
\]
  The fact that this algebra can be presented as an Ore extension was proved in \cite{GaddisYee2025}.
\begin{lemma}[{\cite[Lemma 2.3]{GaddisYee2025}}]\label{B_q(f)ore}
Let $f \in \K[t]$. Then $B_q(f)$ is the Ore extension
\[
B_q(f) = \K_q[u,v]\,[w;\sigma,\delta],
\]
where $\K_q[u,v] = \K\langle u,v \mid uv - qvu\rangle$ and
\[
\sigma(u) = qu, \quad \sigma(v) = q^{-1}v, \quad 
\delta(u) = f(v), \quad \delta(v) = f(u).
\]
\end{lemma}

\begin{proof}
The algebra $\K_q[u,v]$ is presented by generators $u, v$ and the single relation 
$r := uv - qvu = 0$. For a $\K$-linear map $\delta$ defined on the generators to 
extend to a well-defined $\sigma$-derivation on the quotient $\K_q[u,v]$, it is 
necessary and sufficient that both $\sigma$ and $\delta$ respect the relation $r$.

\begin{itemize}
\item[(i)] Computing in the free algebra and passing to the quotient:
\[
\sigma(uv - qvu) = \sigma(u)\sigma(v) - q\,\sigma(v)\sigma(u)
= (qu)(q^{-1}v) - q(q^{-1}v)(qu) = uv - qvu = 0.
\]
Hence $\sigma$ is well defined on $\K_q[u,v]$.

\item[(ii)] Using the twisted Leibniz rule $\delta(ab) = \sigma(a)\delta(b) + 
\delta(a)b$, and substituting $\delta(u) = f(v)$ and $\delta(v) = f(u)$, we 
obtain
\[
\delta(uv) = \sigma(u)f(u) + f(v)v = qu\,f(u) + f(v)v,
\]
\[
\delta(vu) = \sigma(v)f(v) + f(u)u = q^{-1}v\,f(v) + f(u)u.
\]
Therefore,
\begin{align*}
\delta(uv - qvu)
&= \delta(uv) - q\,\delta(vu)\\
&= \bigl(qu\,f(u) + f(v)v\bigr) - q\bigl(q^{-1}v\,f(v) + f(u)u\bigr)\\
&= \underbrace{qu\,f(u) - q\,f(u)u}_{= 0}
+ \underbrace{f(v)v - v\,f(v)}_{= 0} = 0,
\end{align*}
where the cancellations follow from the fact that $u$ commutes with every 
polynomial in $u$, and $v$ commutes with every polynomial in $v$. Hence $\delta$ 
is well defined on $\K_q[u,v]$.
\end{itemize}
\end{proof}

\begin{proposition}\label{basepbw}
The set
\[
\mathcal{B} := \{\, v^i u^j w^k \mid i,j,k \in \mathbb{N} \,\}
\]
is a $\K$-basis of $B_q(f)$.
\end{proposition}

\begin{proof}
By Lemma~\ref{B_q(f)ore}, $B_q(f)$ is the Ore extension $\K_q[u,v][w;\sigma,\delta]$ 
and, as a $\K$-vector space, is a free left $\K_q[u,v]$-module with basis 
$\{1, w, w^2, \dots\}$. Hence every element can be written uniquely as 
$\sum_{k \ge 0} a_k w^k$ with $a_k \in \K_q[u,v]$. Since the quantum plane 
$\K_q[u,v] = \K\langle u,v \mid uv = qvu\rangle$ has PBW basis 
$\{v^i u^j \mid i,j \in \mathbb{N}\}$, every element of $B_q(f)$ can be written 
uniquely as a finite $\K$-linear combination of monomials $v^i u^j w^k$, and 
therefore $\mathcal{B}$ is a $\K$-basis of $B_q(f)$.
\end{proof}

\begin{proposition}\label{prop1}
Let $\K$ be a field, $q \in \K^*$, and $f \in \K[t]$ with $\deg f \le 1$. Then 
$B_q(f)$ is a skew PBW extension of $\K$ of the form 
$\sigma(\K)\langle v, u, w\rangle$.
\end{proposition}

\begin{proof}
Since $\mathsf{deg}(f) \le 1$, there exist $a, b \in \K$ such that $f(t) = at + b$. The 
defining relations of $B_q(f)$ become
\begin{equation}\label{eq:rels-linear}
vu = q^{-1}uv, \qquad wu = quw + av + b, \qquad wv = q^{-1}vw + au + b.
\end{equation}
We verify the conditions of Definition \ref{pbwt}:
\begin{itemize}
\item[(i)] By Proposition~\ref{basepbw}, $B_q(f)$ has PBW basis 
$\mathcal{B} = \{v^i u^j w^k \mid i,j,k \in \mathbb{N}\}$.

\item[(ii)] For all $r \in \K$ and $x_i \in \{v,u,w\}$ we have $x_i r = rx_i$, 
since $\K$ is the base field. Hence $c_{i,r} = r$ and $x_i r - c_{i,r}x_i = 0 \in \K$.

\item[(iii)] From \eqref{eq:rels-linear}:
\[
vu - q^{-1}uv = 0 \in \K \subseteq \K + \K u + \K v + \K w,
\]
\[
wu - q\,uw = av + b \in \K + \K v \subseteq \K + \K u + \K v + \K w,
\]
\[
wv - q^{-1}vw = au + b \in \K + \K u \subseteq \K + \K u + \K v + \K w.
\]
\end{itemize}
Therefore, $B_q(f) \cong \sigma(\K)\langle v, u, w\rangle$ is a skew PBW 
extension of $\K$.
\end{proof}

\begin{remark}
Let $\K$ be a field, $q \in \K^*$, and $b \in \K$. Then the algebra $B_q(b)$ is a 
skew PBW extension of the form $\sigma(\K[v])\langle u, w\rangle$.
\end{remark}

\begin{proposition}\label{prop:PI-Bqf-charp}
Let $\K$ be a field of positive characteristic $p > 0$, let $q \in \K^{*}$, and 
let $f \in \K[t]$. The algebra $B_q(f)$ satisfies the PI property if and only if 
$q$ is a root of unity. In that case, $B_q(f)$ is a finitely generated module 
over $Z(B_q(f))$, the center $Z(B_q(f))$ is a normal affine $\K$-domain, and
\[
\operatorname{gldim}(B_q(f)) = \operatorname{GKdim}(B_q(f)) = 3.
\]
\end{proposition}

\begin{proof}
By Lemma \ref{B_q(f)ore}, $B_q(f) \cong \K[u][v;\sigma_2,0][w;\sigma_3,\delta_3]$, 
where $\sigma_2(u) = qu$, $\delta_2 = 0$, $\sigma_3(u) = qu$, 
$\sigma_3(v) = q^{-1}v$, $\delta_3(u) = f(v)$, and $\delta_3(v) = f(u)$. We 
apply Theorem \ref{Iteradas-3.3} with $R(1) := \K[u]$ and 
$R(2) := \K\langle u,v\rangle = \K_q[u,v]$.

Suppose that $B_q(f)$ satisfies the PI property. By 
Theorem \ref{Iteradas-3.3}-(i), the restriction $\sigma_2|_{Z(R(1))}$ has finite 
order. Since $Z(\K[u]) = \K[u]$ and $\sigma_2^m(u) = q^m u$, the condition 
$\sigma_2^m = \mathrm{id}|_{\K[u]}$ is equivalent to $q^m = 1$. Hence $q$ is a 
root of unity.

Conversely, suppose that $q$ is a root of unity and let $\ell \ge 1$ be such 
that $q^\ell = 1$. Then $\sigma_2^\ell(u) = q^\ell u = u$, so 
$\sigma_2|_{Z(R(1))}$ has finite order. Moreover, $\sigma_3^\ell(u) = q^\ell u = u$ 
and $\sigma_3^\ell(v) = q^{-\ell}v = v$, so $\sigma_3^\ell$ fixes the generators 
of $R(2) = \K_q[u,v]$ and, in particular, fixes every element of $Z(R(2))$. 
Hence $\sigma_3|_{Z(R(2))}$ has finite order. Theorem~\ref{Iteradas-3.3} then 
implies that $B_q(f)$ satisfies the PI property, and part (ii) of the same 
theorem yields the stated properties of the center and the homological dimensions.
\end{proof}

It is noteworthy that, over a field of positive characteristic, the criterion for 
$B_q(f)$ to satisfy the PI property depends only on the parameter $q$ and is 
independent of the polynomial $f$. In contrast, the characteristic-zero case, 
studied in Proposition~3.11 of \cite{GaddisYee2025}, requires additional 
conditions on $f$: besides $q$ being a primitive $\ell$-th root of unity with 
$\ell > 1$, if $f(t) = \sum_{j \ge 0} c_j t^j$ and 
$\mathrm{supp}(f) := \{j \ge 0 \mid c_j \neq 0\}$, one must have 
$\ell \nmid (j+1)$ for all $j \in \mathrm{supp}(f)$.

\begin{example}\label{ej:PI-charp-noPI-char0}
Take $f(t) = t$, so that $\mathrm{supp}(f) = \{1\}$, and $q = -1$, which is a 
root of unity of order $\ell = 2$.
\begin{enumerate}[(1)]
    \item \emph{Positive characteristic.} If $\K = \mathbb{F}_3$, then 
    $B_{-1}(t)$ satisfies the PI property.
    
    \item \emph{Characteristic zero.} If $\K = \mathbb{Q}$, then $q = -1$ 
    is a primitive root of unity of order $\ell = 2$. However, for 
    $j = 1 \in \mathrm{supp}(f)$ one has $j+1 = 2$ and $\ell = 2$ divides $2$, 
    so the additional condition fails and $B_{-1}(t)$ does not satisfy the PI 
    property.
\end{enumerate}
\end{example}
\bibliographystyle{plain}
\bibliography{biblio}
\bigskip
\noindent
J.G.:\\
Escuela de Matemáticas y Estadística, Universidad Pedagógica y Tecnológica de Colombia,\\
Tunja, Colombia\\
james.gomez@uptc.edu.co\\
\\
C.G.:\\
Departamento de Matemáticas, Pontificia Universidad Javeriana,\\
Bogotá, Colombia\\
gallegoj.cm@javeriana.edu.co

\end{document}